\documentclass{amsart}
\usepackage{graphicx} 

\usepackage{amsmath}
\usepackage{amssymb}
\usepackage{amsthm}
\usepackage{hyperref}
\usepackage{color}
\usepackage{a4wide}
\usepackage{dsfont}
\usepackage{mathdots}
\usepackage{tikz-cd}
\usepackage[all]{xy}
\usepackage{pgf,tikz,pgfplots}
\makeatletter

\newcommand{\fraka}{\mathfrak{a}}

\newcommand{\CC}{\mathbb{C}}

\newcommand{\NN}{\mathbb{N}}

\newcommand{\RR}{\mathbb{R}}
\newcommand{\ZZ}{\mathbb{Z}}
\newcommand{\calC}{\mathcal{C}}
\newcommand{\calD}{\mathcal{D}}

\newcommand{\calF}{\mathcal{F}}

\newcommand{\calL}{\mathcal{L}}

\newcommand{\calS}{\mathcal{S}}

\newcommand{\calV}{\mathcal{V}}
\newcommand{\calW}{\mathcal{W}}

\newcommand{\0}{{\bf 0}}
\newcommand{\N}{\mathbb{N}}
\newcommand{\Z}{\mathbb{Z}}

\newcommand{\R}{\mathbb{R}}
\newcommand{\C}{\mathbb{C}}

\DeclareMathOperator{\GL}{GL}

\DeclareMathOperator{\Ind}{Ind}

\DeclareMathOperator{\DSBO}{DSBO}

\DeclareMathOperator{\Hom}{Hom}

\DeclareMathOperator{\rest}{rest}

\DeclareMathOperator{\sgn}{sgn}

\DeclareMathOperator{\diag}{diag}
\DeclareMathOperator{\supp}{supp}

\renewcommand\Re{\operatorname{Re}}

\theoremstyle{plain}
\newtheorem{theorem}{Theorem}[section]
\newtheorem{proposition}[theorem]{Proposition}
\newtheorem{lemma}[theorem]{Lemma}
\newtheorem{corollary}[theorem]{Corollary}

\newtheorem{theoremIntro}{Theorem}

\theoremstyle{definition}
\newtheorem{definition}[theorem]{Definition}
\newtheorem{example}[theorem]{Example}

\newtheorem{remark}[theorem]{Remark}

\newcommand{\ql}[1]{\textcolor{red}{QL: #1}}

\title{Differential symmetry breaking operators for the pair $(\GL_{n+1}(\RR),\GL_n(\RR))$ }

\author{Jonathan Ditlevsen}
\address{Graduate School of Mathematical Sciences, The University of Tokyo, 3-8-1 Komaba, Meguro-ku, Tokyo 153-8914, Japan}
\email{JonathanDitlevsen@gmail.com}

\author{Quentin Labriet}
\address{Centre de recherches math\'ematiques, Universit\'e de Montr\'eal, P.O. Box 6128, Centre-ville Station, Montr\'eal (Qu\'ebec), H3C 3J7}
\email{quentin.labriet@umontreal.ca}

\pgfplotsset{compat=1.18}
\begin{document}

\begin{abstract}
    \noindent In this article we study differential symmetry breaking operators between principal series representations induced from minimal parabolic subgroups for the pair $(\GL_{n+1}(\R),\GL_n(\R))$. Using the source operator philosophy we construct such operators for generic induction parameters of the representations and establish that this approach yields all possible operators in this setting.  We show that these differential operators occur as residues of a family of symmetry breaking operators that depends meromorphically on the parameters. Finally, in the $n=2$ case we classify and construct all differential symmetry breaking operators for any parameters, including the non-generic ones. 
\end{abstract}
\maketitle
{\noindent \raggedright \small\textit{Keywords and phrases:} Branching problems, Symmetry breaking, Principal series representations, Intertwining operators, Source operator method.\\}
\vspace{1em}
\section{Introduction}

Intertwining operators are fundamental objects in representation theory, with their construction and study tracing back to the origins of the field. In the modern representation theory of reductive groups, one of the central themes is the so-called branching problem. This problem focuses on understanding the restriction \( \pi|_H \) of an irreducible representation \( \pi \) of a group \( G \) to a subgroup \( H \). In this framework, a key objective is to construct and classify intertwining operators between \( \pi|_H \) and irreducible representations \( \tau \) of \( H \). These operators are commonly referred to as \textit{symmetry breaking operators}, a term introduced by T. Kobayashi \cite {Kobayashi15}.

Explicit expressions for symmetry breaking operators depend on the chosen models for \( \pi \) and \( \tau \). If these representations are realized on spaces consisting of smooth or holomorphic functions, an interesting class of symmetry breaking operators consists of those which are differential operators. These are called \emph{differential symmetry breaking operators} and the main objective of this article is to construct explicit examples of such operators. A well-known example of a family of differential symmetry breaking operators is the celebrated family of Rankin–Cohen brackets. Introduced first in \cite{Cohen} to classify bi-differential operators acting on modular forms, they can be interpreted as differential symmetry breaking operators for the tensor product of two holomorphic discrete series of $SL_2(\R)$(see \cite{Dijk_Pev,KP_SBO}).

In a series of papers \cite{KobayashiPevzner16_1,KP_SBO}, T. Kobayashi and M. Pevzner developed a general approach, known as the F-method, to study and explicitly construct differential symmetry breaking operators. For this approach, they relate the symbol of differential symmetry breaking operators to the polynomial solutions of certain partial differential equations. They also prove a duality theorem, which identifies the space of differential symmetry breaking operators with a space of homomorphisms between generalized Verma modules. This insight reveals that the problem of constructing differential symmetry breaking operators is algebraic in nature. The F-method became a standard approach to classify differential symmetry breaking operators and has been successfully applied in various settings (see for example \cite{KP_SBO, KobayashiKuboPevzner16,KobayashiSpeh-I-15, KobayashiSpeh18,FrahmWeiske2020, KuboOrsted2024,P-V23}).

In this paper, we study differential symmetry breaking operators between principal series representations induced from minimal parabolic subgroups for the pair \((G, H) = (\mathrm{GL}_{n+1}(\mathbb{R}),\allowbreak \mathrm{GL}_n(\mathbb{R}))\). These operators have been fully classified for all finite multiplicity pairs $(G,H)$ with $ \operatorname{rank}G=1$, for spherical principal series representations, thanks to the F-method see e.g. \cite{KobayashiSpeh-I-15,KobayashiSpeh18,FrahmWeiske2020}. 
For higher rank groups, only some special cases have been studied so far for non-degenerate principal series (e.g. \cite{BenSaidClercKoufany2020,KuboOrsted2024,FrahmWeiske19}), since the problem becomes more sophisticated. One initial reason is that the system of differential equations that arises gets more complicated. In this paper we exhibit another phenomenon which explains the complexity of the higher rank case. This is due to the existence of several different embeddings of $H$ into $G$, unlike in the rank one case where there is only one, this then leads to different classes of differential operators (see Section \ref{sec:DSBOGeneral} for the details). As a consequence, the F-method should be applied once for each different embedding. A last difficulty can appear when the nilradical $N$ in the Langlands decomposition of the parabolic subgroup is non-abelian in which case the system of differential equations has order higher than two (see \cite[Proposition 3.10]{KobayashiPevzner16_1}). 

In our setting, $\GL_{n+1}(\R)$ is of rank $n+1$ and the nilradical is the space of upper triangular unipotent matrices which is non abelian if $n>1$. Thus we face all the problems mentioned in the previous paragraph and applying the F-method presents significant difficulties in this context. Thus we adopt an alternative approach known as the \emph{source operator philosophy}, inspired by the work of R. Beckmann and J-L. Clerc \cite{Beck_Clerc}, and used with success in different settings (e.g. \cite{Clerc2017,BenSaidClercKoufany2020}). We talk about a philosophy here since there does not exist a general recipe for it, and no results guaranteeing it works. Using this philosophy, we construct a large class of differential symmetry breaking operators. We are able to show that for `generic' parameters  these are all the differential symmetry breaking operators.

\subsection{Results}
We will now present the source operator philosophy together with the main results of the article. Set $(G,H)=(\GL_{n+1}(\R),\GL_n(\R))$, and consider their respective minimal parabolic subgroup $P_G$ and $P_H$. In this setup $(G/P_G, H/P_H)$ is a pair of real flag manifolds. We consider the principal series representations as smooth normalized induction
\[
\pi_{\xi,\lambda}=\Ind_{P_G}^G(\xi\otimes e^\lambda\otimes 1),\qquad \tau_{\eta,\nu}=\Ind_{P_H}^H(\eta\otimes e^\nu\otimes 1),
\]
where $(\xi,\lambda)\in (\ZZ/2\ZZ)^{n+1}\times \CC^{n+1}$ and $(\eta,\nu)\in (\ZZ/2\ZZ)^n\times \CC^n$.  Our goal is then to construct a basis of the space of differential symmetry breaking operators between $\pi_{\xi,\lambda}|_H$ and $\tau_{\eta,\nu}$.
 
The source operator philosophy can be described as follows (see Section \ref{subsec:SourceOperator} for more details).
First, we classify the polynomials $p$ on $G$ such that $f\mapsto pf$ are $H$-intertwining operators between principal series $\pi_{\xi,\lambda}$ and $\pi_{\xi',\lambda'}$. Then we conjugate these multiplication operators by Knapp--Stein intertwining operators $T_{\xi,\lambda}^w:\pi_{\xi,\lambda}\to \pi_{w(\xi,\lambda)}$ to obtain two different explicit families of $H$-intertwining differential operators, namely $\calD_i$ and $\calF_j$. They will play the role of the source operator in the sense that composition of such operators leads to differential $H$-intertwining operators between principal series representations of $G$. The fact that these operators are differential operators is non trivial and was proved together with their explicit expression by the first author in \cite{DitlevsenFrahm24}. However, viewing these operators as $H$-intertwining and using them to construct differential symmetry breaking operators is new in this context.

Finally, let us call $\calL:\pi_{\xi,\lambda}\to \pi_{\xi',\lambda'}$ a composition of operators $\calD_i$ and $\calF_j$. Then for any symmetry breaking operator $A:\pi_{\xi',\lambda'}\to \tau_{\eta,\nu}$ the composition $A\circ \calL$ is also a symmetry breaking operator. We consider this composition for two different types of symmetry-breaking operators.

\subsection*{The restriction operators}
First, we compose with restriction operators. In $G$ there are exactly $n+1$ cosets $x_kP_G$ such that $P_Hx_kP_G\subseteq x_kP_G$ and each of them defines a symmetry breaking operator $\rest_k$  given by $f\mapsto f(\,\cdot\, x_k)$. Differential symmetry breaking operators will then be of the form $\rest_k\circ D$ for some differential operator $D$ on $G$ and the distributional support of $\rest_k\circ D$ will be $x_kP_G$. This situation differs from the real rank-one case, where only one such map exists. Composing the restriction maps $\rest_k$ with combinations of some of the $\calF_i$ and $\calD_j$ gives differential symmetry breaking operators. However, such construction might lead to zero operators and so for $\alpha\in \NN^n$ we consider the following combinations
\[
\calL_{\alpha,k}=\calF_1^{\alpha_1}\circ\cdots\circ \calF_k^{\alpha_k}\circ\calD_{k+2}^{\alpha_{k+1}}\circ\cdots\circ \calD_{n+1}^{\alpha_n}.
\]
We then show the following theorem

\begin{theoremIntro}
    Let $k\in \{0,\dots,n\}$ and $\lambda \in \C^n$ such that $\lambda_i-\lambda_j\notin \Z$ for all $i,j$. Then the dimension of the space of differential symmetry breaking operators with distributional support $x_kP_G$ is either $0$ or $1$. When it is $1$, there exists $\alpha\in \N^n$ such that a basis of the space of differential symmetry breaking operators is given by 
    \[\rest_k\circ\calL_{\alpha,k}:\pi_{\xi,\lambda}\to \tau_{\eta,\nu},\] 
    with \[\eta=(\xi_1+\alpha_1,\dots,\xi_k+\alpha_k,\xi_{k+2}+\alpha_{k+1},\dots,\xi_{n+1}+\alpha_n),\] and \[\nu=(\lambda_1+\tfrac{1}{2}+\alpha_1,\dots,\lambda_k+\tfrac{1}{2}+\alpha_k,\lambda_{k+2}-\tfrac{1}{2}-\alpha_{k+1},\dots,\lambda_{n+1}-\tfrac{1}{2}-\alpha_{n}).\]
\end{theoremIntro}

In Theorem \ref{thm:genericDSBO} we describe explicitly for which parameters it is of dimension $0$ or $1$, and which parameter $\alpha$ is involved. Parameters $\lambda$ satisfying $\lambda_i-\lambda_j\notin \Z$  will be called \emph{generic} and for these parameters the principal series representations are irreducible. We know that $(G,H)$ is a multiplicity-one pair as shown by Sun-Zhu (see \cite{SunZhu2012}). Thus for generic parameters the dimension is at most $1$. To find when it is of dimension $0$ we look at the distributional kernel associated with a differential symmetry breaking operator and show that in order for it to have the right equivariance properties the parameters must belong in a specific set. When we are not in the multiplicity zero case we prove that $\rest_k\circ \calL_{\alpha,k}\neq 0$  and thus is a basis of the space of differential symmetry breaking operators.

\subsection*{The holomorphic family}
The other type of symmetry breaking operator we compose $\calL_{\alpha,k}$ with is an explicit meromorphic family $A_{\xi,\lambda}^{\eta,\nu}:\pi_{\xi,\lambda}\to \tau_{\eta,\nu}$ that J. Frahm and the first author studied in \cite{DitlevsenFrahm24}. The composition with $\calL_{\alpha,k}$ gives a functional equation of the type $A_{\xi',\lambda'}^{\eta,\nu}\circ\calL_{\alpha,k}=p(\xi,\lambda,\eta,\nu)A_{\xi,\lambda}^{\eta,\nu}$ where $p$ is a polynomial in $\lambda$ and $\nu$. This observation can be used to show:
\begin{theoremIntro}
The differential symmetry breaking operators $\rest_k\circ \calL_{\alpha,k}$ all occurs as residues of the meromorphic family $A_{\xi,\lambda}^{\eta,\nu}$.
\end{theoremIntro}
The proof of this theorem relies on taking residues such that $A_{\xi,\lambda}^{\eta,\nu}=\operatorname{const}\times \rest_k$ and then composing with $\calL_{\alpha,k}$ and using the functional equation. We use this relation to get a sufficient condition for the differential symmetry breaking operators $\rest_k\circ \calL_{\alpha,k}$ to be non-vanishing. In \cite{KobayashiSpeh18} T. Kobayashi and B. Speh introduced the notions of regular and sporadic differential symmetry breaking operators. If a differential symmetry breaking operator occurs as a residue of meromorphic family of symmetry breaking operators it is called regular and if not it is called sporadic. With this terminology we can rephrase Theorem B as saying $\rest_k\circ \calL_{\alpha,k}$ is regular symmetry breaking operator. 

\subsection*{The \texorpdfstring{$n=2$}{n=2} case}
The last part of this article is devoted to the full classification of differential symmetry breaking operators in the $n=2$ case. In this section we do not restrict to generic parameters. To achieve the classification we solve differential equations for the distributional kernel of differential symmetry breaking operators. This allows us to describe explicitly the distributional kernel of such operators for the three different possibilities for their one point support. 
In the process we exhibit a dimension two phenomenon for a specific one point support. Note that this does not contradict the multiplicity one property for the pair $(G,H)$ since one of the involved representations is not irreducible. 

We also investigate to which extent the source operator philosophy accounts for all possible differential symmetry breaking operators. The problem occurs only for parameters which are not generic for which our operators $\rest_k\circ \calL_{\alpha,k}$ might be zero. In case needed, we solve this issue by the use of an extra differential intertwining operator between principal series of $\GL_3(\R)$ (or $\GL_2(\R)$). This leads to the construction of all differential symmetry breaking operators by means of our operators $\rest_k\circ \calL_{\alpha,k}$ even in the non generic situation.

\subsection*{Acknowledgment}
The authors would like to address special thanks to Pr. J.Frahm, Pr. T.Kobayashi and Pr. M.Pevzner for their useful comments and advices on this project. Both authors were partially supported by a research grant from the Villum Foundation (Grant No. 00025373). The authors acknowledge support of the Institut Henri Poincaré (UAR 839 CNRS-Sorbonne Université), and LabEx CARMIN (ANR-10-LABX-59-01). The first author is supported by the Carlsberg Foundation, grant CF24-045.

\subsection*{Notations}

Here are some general notations used in the paper. $\NN=\{0,1,2,\dots\}$. For $n\in \NN$ then $[n]\in \{0,1\}$ denotes the remainder of $n$ after division by 2. $e_i$ denotes the standard basis of $\C^n$ and $\mathds{1}$ is the vector in $\CC^n$ with all entries equal to one.


\section{Preliminaries}
\subsection{Structure and principal series representations for \texorpdfstring{$\GL(n,\R)$}{GL(n,R)}}\label{sec:principal series reps}
\noindent This section introduces the basic definitions and concepts for principal series representations for $\GL_n(\R)$ (see \cite{Knapp01} for more details). 

Let $G=\GL_n(\R)$, and $P_G$ its minimal parabolic subgroup consisting of upper triangular matrices. It admits the Langlands decomposition $P_G=M_GA_GN_G$ where  $M_G$ is the group of diagonal matrices with entries from $\{\pm1\}$ i.e. $M_G\simeq (\Z/2\Z)^n$, $A_G$ is the group of diagonal matrices with strictly positive real entries with Lie algebre $\fraka$ and $N_G$ is the group of unipotent upper triangular matrices. 

Let $\xi$ be a character of $M_G$, and $\lambda \in \fraka_\C^*$. Then $\xi\otimes e^\lambda\otimes 1$ is a one dimensional representation of $P_G=M_GA_GN_G$ where $1$ is the trivial representation of $N_G$. Using smooth normalized parabolic induction we obtain the principal series representation $\pi_{\xi,\lambda}=\Ind_{P_G}^G(\xi\otimes e^\lambda\otimes 1)$ as the left-regular representation of $G$ on 
\[
\{f\in \calC^\infty(G)\,|\,f(gman)=\xi(m)^{-1}a^{-\lambda-\rho}f(g)\,\,\forall man\in M_GA_GN_G\},
\]
with $\rho=\frac{1}{2}(n-1,n-3,\dots,3-n,1-n)$ the half sum of the positive roots. Since $\widehat{M_G}\simeq (\ZZ/2\ZZ)^n$ we get that a character $\xi=(\xi_1,\dots,\xi_{n})\in\{0,1\}^n$ is described by 
\begin{align*}
\text{diag}(\varepsilon_1,\dots,\varepsilon_{n}) \in M_G \mapsto \sgn (\varepsilon_1)^{\xi_1}\dotsb \sgn (\varepsilon_{n})^{\xi_{n}}\in \{-1,1\} .
\end{align*} 
The isomorphism $\mathfrak{a}_{\CC}^*\simeq \C^{n}$ is realized by the map $\lambda\mapsto \big(\lambda(E_{1,1}),\dots, \lambda (E_{n,n})\big)$. Consider the following function defined on $\R^\times$
\[|x|_\xi^\lambda =\sgn(x)^\xi|x|^\lambda. \qquad (\xi\in \ZZ/2\ZZ, \; \lambda \in\CC)
\]
We can then consider the principal series representations as functions $f\in \calC^\infty(G)$ which satisfy
\[
f(gman)=|x_1|^{-\lambda_1-\frac{n-1}{2}}_{\xi_1}\cdots |x_n|^{-\lambda_n-\frac{1-n}{2}}_{\xi_n}f(g),\quad ma=\diag(x_1,\dots,x_n)\in M_G A_G, \, n\in N_G.
\]
The principal series representations are known to be irreducible if for all $1\leq i,j\leq n$ we have
\[\lambda_i-\lambda_j\notin
\begin{cases}
    2\Z+1 & \text{if } \xi_i=\xi_j,\\
    2\Z \backslash \{0\}& \text{if } \xi_i\neq\xi_j.
\end{cases}
\]
See e.g. \cite{Moeglin97}. We say that  $(\xi,\lambda)$ is \emph{generic} if $\lambda_i-\lambda_j\notin \ZZ$ which implies that $\pi_{\xi,\lambda}$ is irreducible. 
\begin{remark}
These representations can also be geometrically realized on the space of smooth sections of the line bundle $G\times_{P_G} (\xi\otimes e^\lambda \otimes 1)$ whose base space is the flag variety $G/P_G$.
\end{remark}
Let $K_G=\operatorname{O}(n)$ be the maximal compact subgroup of $G$ and $W_G:=N_{K_G}(A_G)/Z_{K_G}(A_G)$ be the Weyl-group of $G$ which is isomorphic to the symmetric group on $n$ letters $ \calS_n$. The Weyl group $W_G$ carries an action to $\widehat{M}_G$ by $[w\xi](m)=\xi(\tilde{w}^{-1}m\tilde{w})$ where $w=[\tilde{w}]\in W$ and similarly to $\fraka_\CC^*$ by $[w\lambda](a)=\lambda(\tilde{w}^{-1}a\tilde{w})$. This is simply the permutation of the components of $\xi$ and $\lambda$ under $w\in \calS_n$. Abusing notation, we will not distinguish between elements of $W$, $\calS_n$ and the permutation matrices corresponding to elements of $\calS_n$.

There exists a unique (up to scalar multiples) $G$-invariant integral $d(gP_G)$ on the space of sections of the bundle $G\times_{P_G}\CC_{2\rho_G}\to G/P_G$, where $\CC_{2\rho_G}=1\otimes e^{2\rho_G}\otimes 1$ is the one-dimensional representation corresponding to the modular function of $P_G$ (see e.g. \cite{CiobotaruFrahm23}). This integral gives rise to a $G$-invariant pairing 
\[
\langle\,\cdot\,,\,\cdot\,\rangle_{G/P_G}:\pi_{\xi,\lambda}\times \pi_{\xi,-\lambda}\to \CC,\qquad (f_1,f_2)\mapsto \int_{G/P_G}f_1(g)f_2(g)\,d(gP_G).
\]
Let $\overline{N}_G$ be the opposite unipotent radical which consists of the lower triangular unipotent matrices. Then on the open dense Bruhat cell $\overline{N}_GM_GA_GN_G$ this pairing can also be written as 
\[
\langle f_1,f_2\rangle_{G/P_G}=\int_{\overline{N}_G}f_1(\overline{n})f_2(\overline{n})\,d\overline{n},
\]
with $d\overline{n}$ being the suitably normalized Haar measure on $\overline{N}_G$ see  \cite[equation (5.25)] {Knapp01}.


\subsection{The Knapp--Stein intertwining operators}\label{sec:Knapp-Stein}
For every $w=[\tilde{w}]\in W_G$ the integral
\[
T_{\xi,\lambda}^wf(g)=\int_{\overline{N}_G\cap \tilde{w}^{-1}N_G\tilde{w}}f(g\tilde{w}\overline{n})\,d\overline{n},
\]
converges absolutely in some range of $\fraka_\CC^*$ and defines an intertwining operator $\pi_{\xi,\lambda}\to \pi_{w (\xi,\lambda)}$ known as the Knapp--Stein intertwining operator, see e.g. \cite{Knapp01} for a thorough exposition. It can be shown that the family of operators $T_{\xi,\lambda}^w$ can be meromorphically extended to all $\lambda\in \fraka_\CC^*$ (if viewed as an intertwining operator in the compact picture where the representation space is
independent of $\lambda$). These operators satisfies that
\begin{equation}\label{eq:Knapp-SteinProduct}
	T^{w'w}_{\xi,\lambda}=T^{w'}_{w(\xi,\lambda)}\circ T^{w}_{\xi,\lambda}\qquad (w,w'\in W_G)
\end{equation}
when $\ell(w'w)=\ell(w')+\ell(w)$, where $\ell$ denotes the length of an element in $W_G$. Here the length $\ell(w)$ corresponds to the minimal number of simple transpositions $w_i=(i \ i+1)$ needed to write a permutation $w$  in $\calS_n$. The Knapp--Stein intertwining operators satisfy the following relation:
\begin{equation}\label{eq:InverseTony}
T_{w_i(\xi,\lambda)}^{w_i}\circ T_{\xi,\lambda}^{w_i}=c_i(\xi,\lambda)\operatorname{id},
\end{equation}
where the scalar $c_i(\xi,\lambda)$ is the Harish--Chandra's $c$-function given by
\begin{equation}\label{eq:c-function} 
c_i(\xi,\lambda)=(-1)^{\xi_i+\xi_{i+1}}\pi\frac{\Gamma(\frac{\lambda_i-\lambda_{i+1}+[\xi_i+\xi_{i+1}]}{2})\Gamma(\frac{\lambda_{i+1}-\lambda_i+[\xi_i+\xi_{i+1}]}{2})}{\Gamma(\frac{\lambda_i-\lambda_{i+1}+1+[\xi_i+\xi_{i+1}]}{2})\Gamma(\frac{\lambda_{i+1}-\lambda_i+1+[\xi_i+\xi_{i+1}]}{2})}.
\end{equation}
More generally if $\sigma=w_{i_1}\dots w_{i_{\ell(\sigma)}}$ then
\[
T^{\sigma^{-1}}_{\sigma(\xi,\lambda)}\circ T^{\sigma}_{\xi,\lambda}=c_\sigma(\xi,\lambda)\times\operatorname{id}
\]
and the $c$-function can be computed using the Gindikin--Karpelevich formula (see e.g. \cite{Kna03}) as
\begin{equation}\label{eq:Gindikin--Karpelevich}
c_\sigma(\xi,\lambda)=\prod_{j=1}^{\ell(\sigma)}c_{i_j}\big(w_{i_{j-1}}\cdots w_{i_1}(\xi,\lambda)\big).
\end{equation}
For a simple transposition $w_i$ we have the following simple form of the Knapp--Stein operator:
\[
T_{\xi,\lambda}^{w_i}f(g)=\int_\RR |x|^{\lambda_i-\lambda_{i+1}-1}_{\xi_i+\xi_{i+1}}f(g\overline{n}_i(x))\,dx
\]
where $\overline{n}_i(t)=e^{tE_{i+1,i}}$ see e.g. \cite{DitlevsenFrahm24}. From classical theory of the Riesz distribution see e.g. \cite[Chapter 3.]{GelfandShilov64} we normalize the Knapp--Stein intertwiner by $\mathbf{T}_{\xi,\lambda}^{w_i}:=\Gamma(\frac{\lambda_i-\lambda_{i+1}+[\xi_i+\xi_{i+1}]}{2})^{-1}T_{\xi,\lambda}^{\eta,\nu}$ which makes it holomorphic and at the poles it becomes a differential operator given as
\begin{align*}
\mathbf{T}_{\xi,\lambda}^{w_i}f(g)\big |_{\lambda_i-\lambda_{i+1}+[\xi_i+\xi_{i+1}]=-2j}&=\frac{(-1)^{j+[\xi_i+\xi_{i+1}]}(j+[\xi_i+\xi_{i+1}])!}{(2j+[\xi_i+\xi_{i+1}])!}\int_\RR \delta^{(\lambda_{i+1}-\lambda_i)}(x)f(g\overline{n}_i(x))\,dx\\
&=(-1)^j \frac{(j+[\xi_i+\xi_{i+1}])!}{(2j+[\xi_i+\xi_{i+1}])!}(\varepsilon^{i+1,i})^{\lambda_{i+1}-\lambda_i}f(g),
\end{align*}
where $\varepsilon^{i+1,i}$ is defined in Section \ref{subsec:SourceOperator}.


\subsection{Symmetry breaking operators between principal series}
Let $G=\GL_{n+1}(\RR)$ and $H=\GL_n(\RR)$ and denote by $P_G=M_GA_GN_G$ and $P_H=M_HA_HN_H$ the minimal parabolic subgroup of $G$ and $H$. Write $\pi_{\xi,\lambda}$ $(\xi\in (\ZZ/2\ZZ)^{n+1}$, $\lambda\in \CC^{n+1})$ for the principal series of $G$ and $\tau_{\eta,\nu}$ $(\eta\in (\ZZ/2\ZZ)^n$, $\nu\in \CC^n)$ for the principal series of $H$. Let $\mathcal{D}'(G)$ denote the space of distributions on $G$. Following \cite{Frahm23} 
the space of symmetry breaking operators can be identified 
\[
\Hom_H(\pi_{\xi,\lambda},\tau_{\eta,\nu})\simeq \calD'(G)_{\xi,\lambda}^{\eta,\nu},
\]
where
\begin{multline}\label{eq:EquivarianceKernel}
	\calD'(G)_{\xi,\lambda}^{\eta,\nu} = \{K\in\calD'(G):K(m_Ha_Hn_Hgm_Ga_Gn_G)=\xi(m_G)\eta(m_H)a_G^{\lambda-\rho_G}a_H^{\nu+\rho_H}K(g)\\\mbox{for all }m_Ga_Gn_G\in P_G,m_Ha_Hn_H\in P_H,g\in G\},
\end{multline}
in the sense that a distribution $K\in\calD'(G)_{\xi,\lambda}^{\eta,\nu}$ defines a symmetry breaking operator $A\in\Hom_H(\pi_{\xi,\lambda},\tau_{\eta,\nu})$ by
\begin{equation}
    Af(h) = \int_{G/P_G} K(h^{-1}g)f(g)\, d(gP_G)=\langle K,f(h\,\cdot \,)\rangle_{G/P_G}. \qquad (f\in\pi_{\xi,\lambda},h\in H).\label{eq:SBOinTermsOfKernel}
\end{equation}
Here, the integral has to be understood in the distribution sense using the unique $G$-invariant integral. The support of such a distributional kernel needs to be a $P_H$-invariant set in $G/P_G$ see \cite[Chapter 3.]{KobayashiSpeh15}. 

\subsection{\texorpdfstring{$p$}{p}-differential operators}\label{sec:DSBOGeneral}
In this section, we recall the definition of a differential operator between two manifolds from \cite{KobayashiPevzner16_1} and describe the possible support for the kernel of these operators in the context of homogeneous spaces. We fix $X$ and $Y$ two smooth manifolds, $\mathcal{V}$ and $ \mathcal{W}$ two vector bundles over $X$ and $Y$ respectively. We denote $\mathcal{C}^\infty(X,\mathcal{V})$ the space of smooth sections on $\mathcal{V}$. 

\begin{definition}
Let $T : \mathcal{C}^\infty(X,\mathcal{V}) \to \mathcal{C}^\infty(Y,\mathcal{W})$ be a linear operator. If $p:Y\to X$ is a smooth map then $T$ is a \emph{$p$-differential operator} if for any $f\in \mathcal{C}^\infty(X,\mathcal{V})$
    \begin{equation}
        p(\supp Tf)\subset \supp f
    \end{equation} 
    where $\supp f$ denotes the support of a function $f$.
\end{definition}

\begin{remark}[{\cite[Example 2.4]{KobayashiPevzner16_1}}]
    Let $p:Y\to X$ be an immersion. Choose an atlas of local coordinates $(y_i, z_j )$ on $X$ in such a way that $y_i$ form an atlas on $Y$ . Then, every $p$-differential operator $T$ between $\calV$ and $\calW$ is a finite sum of the form
\[\sum_{\alpha,\beta} g_{\alpha,\beta} (y) \frac{\partial^{\alpha+\beta}}{\partial_y^\alpha\partial_z^\beta}, \]
where $g_{\alpha,\beta}(y)$ are smooth functions on $Y$ with values in the operators between the fibers of $\calV$ and $\calW$. This justifies the terminology "differential operators". 
\end{remark}

To a linear operator $T : \mathcal{C}^\infty(X,\mathcal{V}) \to \mathcal{C}^\infty(Y,\mathcal{W})$ we associate its distributional kernel $K_T\in \calD'(X\times Y)$. So that for $f\in \mathcal{C}^\infty(X,\mathcal{V}) $ we have
\[ Tf(y)=\langle K_T(y), f\rangle. \]
\begin{proposition}[{see \cite{KobayashiPevzner16_1}}]
A linear operator $T:\mathcal{C}^\infty(X,\mathcal{V}) \to \mathcal{C}^\infty(Y,\mathcal{W})$ is a $p$-differential operator if and only if its support $\supp (K_T)$ satisfies \[\supp (K_T)\subset \Delta(Y)=\{(p(y),y)\in X\times Y | y\in Y\}.\]
\end{proposition}
Let $G$ and $H$ be Lie groups. From now on, we assume that $\mathcal{V}$ is a $G$-equivariant vector bundle with respect to the action of a Lie group $G$ and that both $\calV$ and $\calW$ are $H$-equivariant. The following Corollary describes what $p$ maps one can expect when the base spaces are homogeneous spaces. 
\begin{corollary} \label{cor:p-differential}
Let $X=G/G' $ and $Y=H/H'$ be two homogeneous spaces. If there exists a non-zero $p$-differential operator operator between $\calC^\infty(X,\calV)$ and $\calC^\infty(Y,\calW)$ then the map $p:Y\to X$ is $H$-intertwining. If this is the case then $p$ is defined by $p(eH')=y_0 G'$ with $y_0$ satisfying $H'y_0 G'\subset y_0 G'$. 
\end{corollary}
\begin{proof}
Let $T$ be a non zero $p$-differential $H$-intertwining operator. The support of $K_T$ is then $H$-invariant and thus $\supp(K_T)=\Delta(Y)$ since $Y$ is homogeneous. Then $p$ satisfies $p(h\cdot y)=h\cdot p(y)$ for all $y\in Y$. Indeed, $h\cdot (p(y),y)=(h\cdot p(y),h\cdot y)=(p(h\cdot y),h\cdot y)$. Finally, $p$ is uniquely determined by the image $p(eH')=y_0G'$ for some $y_0\in G'$, and it has to satisfy $H'y_0G'\subset y_0G'$ for $p$ to be well defined. 
\end{proof}
\begin{remark}
The converse of the first statement in Corollary \ref{cor:p-differential} is true if there exists a non zero $H$-intertwining map $F$ between the fiber of $\calV$ and the fiber of $\calW$. If this is the case, the map $f\mapsto F\circ f\circ p$ is a non zero $p$-differential $H$-intertwining operator. 
\end{remark}
Assume now that $T$ is a $H$-intertwining operator then there exists a distribution kernel $\tilde{K}_T\in \calD'(X)$ such that
\[
Tf(y) =\langle \tilde{K}_T, h\cdot f\rangle=\langle \tilde{K}_T,f(h^{-1}\cdot)\rangle , 
\]
with $y=h^{-1}H'$. Thus we have $\tilde{K}_T=K_T(eH')$.
\begin{corollary}
Consider the homogeneous spaces $X=G/G' $ and  $Y=H/H'$ and let $T$ be a  non zero $H$-intertwining $p$-differential operator. Then its support has only one point, more precisely $\supp(\tilde{K}_T)={y_0G'}$ where $y_0$ satisfies $H'y_0G'\subset y_0G'$. 
\end{corollary}
\begin{proof}
Indeed, $y\in \supp (\tilde{K}_T)$ if and only if $(y,eH')\in \supp(K_T)$. Hence $y=p(eH')=y_0G'$.
\end{proof}
We now focus on our pair of groups of interest $(G,H)=(\GL_{n+1}(\R),\GL_n(\R))$ where $H$ is embedded in $G$ via the map
$$ h\mapsto \begin{pmatrix}
    h&0\\ 0&1
\end{pmatrix}.$$
We also consider the associated homogeneous spaces $X=G/P_G$ and $Y=H/P_H$. 

\begin{proposition} \label{prop:One-point supports}
The cosets $xP_G \in G/P_G$ such that $P_H x P_G\subset x P_G$ are of the form $x_kP_G$ where
\[
x_k=\begin{pmatrix}
    I_{k}&0& 0\\
    0& 0& I_{n-k}\\
    0&1&0
\end{pmatrix},\text{ for } 0\leq k \leq n.
\]
\end{proposition}
\begin{proof}
This is seen by considering $xP_G$ as a complete flag which should be invariant under the action of $P_H$. 
\end{proof}
Realizing $\pi_{\xi,\lambda}$ as smooth sections of the line-bundle $\calV_{\xi,\lambda}:=G\times_{P_G}(\xi\otimes e^\lambda\otimes 1)$ and $\tau_{\eta,\nu}$ as smooth sections of the line-bundle $\calW_{\eta,\nu}=H\times_{P_H}(\eta\otimes e^\nu\otimes 1)$ we can talk about $H$-intertwining $p$-differential operators between $\calC^\infty(G/P_G,\calV_{\xi,\lambda})$ and $\calC^\infty(H/P_H,\calW_{\eta,\nu})$. By Proposition \ref{prop:One-point supports} and Corollary \ref{cor:p-differential} there are exactly $n+1$ of such $p$'s. We denote these maps by $p_k$ where $p_k$ corresponds to the point $x_k$. 

\begin{definition}
We denote by $\operatorname{DSBO}_k(\pi_{\xi,\lambda},\tau_{\eta,\nu})$ the space of all $p_k$-differential symmetry breaking operators.
\end{definition}

\begin{example} \label{ex:restrictionMap}
 The restriction maps
 \[
 \rest_k: \pi_{\xi,\lambda}\to\tau_{\eta_k(\xi),\nu_k(\lambda)},\qquad  f\mapsto (h\mapsto f(hx_k))
 \]
are differential symmetry breaking operators i.e. they intertwine the $H$-action where $\nu_k(\lambda):=(\lambda_1+\frac{1}{2},\cdots,\lambda_{k}+\frac{1}{2},\lambda_{k+2}-\frac{1}{2},\cdots,\lambda_{n+1}-\frac{1}{2})$ and $\eta_k(\xi):=(\xi_1,\cdots,\xi_{k},\xi_{k+2},\cdots, \xi_{n+1})$. Notice that for $k=n$ then $x_n=I_{n+1}$ and so $\rest_n$ is the restriction to $H$ in the usual sense. It is clear that $\rest_k$ intertwines the $H$-action so we have to check that it goes between the right function spaces. We have
\[ 
f(hm_Ha_Hn_Hx_k)=f(hx_kx_k^{-1}m_Ha_Hx_k x_k^{-1}n_Hx_k)=(x_k^{-1}m_Ha_Hx_k)^{-\lambda-\rho_G}f(hx_k).
\]
Indeed, one can check for 
$$
n_H=\begin{pmatrix}
    N_k&A&0\\0&N_{n-k}&0\\0&0&1
\end{pmatrix}\in N_H\qquad (N_k\in \RR^{k\times k}, N_{n-k}\in \RR^{(n-k)\times (n-k)}),
$$
that 
$$
x_k^{-1}n_Hx_k=\begin{pmatrix}
    N_k&0&A\\
    0&1&0\\
    0&0&N_{n-k}
\end{pmatrix}\in N_G.
$$
Finally, one can check that the parameters fit, giving us that the maps $\rest_k$ are in $\DSBO_k(\pi_{\xi,\lambda},\tau_{\eta_k(\xi),\allowbreak \nu_k(\lambda)})$.
\end{example}
\begin{remark}
For $\lambda=-\rho_G$, we have 
\[
\rest_k\in\operatorname{Hom}_H(\pi_{\xi,\lambda},\tau_{\eta_k(\xi),\nu_k(\lambda)}),\ \text{ for all } 0\leq k\leq n.
\]
This shows that even though for generic parameters $(\xi,\lambda,\eta,\nu)$ the multiplicity $\dim \Hom_H(\pi_{\xi,\lambda},\tau_{\eta,\nu})$ is at most one it can easily be more than one for non-generic parameters. Here it is at least $n+1$.
\end{remark}
\begin{remark}
The maps $\rest_k$ are equivalent to the restriction map for another embedding of $\GL_n(\RR)$ into $ \GL_{n+1}(\R)$. More precisely, consider the following embeddings 
$$
\iota_k:H\rightarrow G,\quad h\mapsto x_k^{-1}\begin{pmatrix}
        h&0\\0&1
    \end{pmatrix}x_k. 
$$
Notice that our choice of embedding for $H$  corresponds to $k=n$. Then the map $\Theta_k: f\mapsto f(x_k\cdot)$ intertwines the action of $H$ corresponding to the different embedding, i.e.
    \[
    \Theta_k\left(\iota_n(h)\cdot f\right)=\iota_k(h)\cdot\Theta_k f. 
    \] 
To conclude this remark, notice we have $\rest_k=R_k\circ~ \Theta_k$ where $R_k$ denotes the restriction to $\iota_k(H)$. 
    
\end{remark}

\section{Differential symmetry breaking operators using the source operator}
In this section we find an explicit family of differential symmetry breaking operators by conjugating multiplication operators by standard intertwining operators.
\subsection{Intertwining multiplication operators}
In this subsection we want to find polynomials $q$ such that the multiplication map $M_q:f\mapsto qf$ is $H$-intertwining between the principal series representations $\pi_{\xi,\lambda}$ of $G$. 
\begin{lemma}\label{lemma:multiplicationIntertwining}
If a polynomial $q$ on $G$ satisfies 
\begin{enumerate}
    \item[(P1)] There exists $\varepsilon\in \widehat{M}_G$ and $\mu \in\fraka_G^*$ such that $q(gman)=a^{\mu}\varepsilon(m)^{-1}q(g)$ for $man\in M_GA_GN_G$ i.e. $q\in \pi_{\varepsilon,-\mu-\rho_G}$,
    \item[(P2)]  There exists a character $\chi$ of $H$ such that $q(hg)=\chi(h)q(g)$ for all $h\in H$,
\end{enumerate}
then
\[
M_q\in \Hom_H(\pi_{\xi,\lambda},\pi_{\xi+\varepsilon,\lambda-\mu}\otimes \chi).
\]
\end{lemma}
\begin{proof}
The first conditions ensures we land in the right space and the second conditions ensures that the map is intertwining.
\end{proof}
Consider the polynomials 
\[
\Phi(g)=g_{n+1,1},\qquad \Psi(g)=\det(g_{ij})_{1\leq i,j\leq n} \qquad (g\in G),
\]
which satisfy \emph{(P1)} and \emph{(P2)} as
\begin{equation*}\label{eq:Pequivariance}
\Phi(gman)=m_1a_1\Phi(g),\qquad \Psi(gman)=m_1a_1\cdots m_na_n\Psi(g)\qquad (man=\diag(m_ia_i)n\in P_G),
\end{equation*}
\begin{equation*}\label{eq:Hequivariance}
\Phi(hg)=\Phi(g),\qquad \Psi(hg)=\det(h)\Psi(g)\qquad (h\in H).
\end{equation*}
So by Lemma \ref{lemma:multiplicationIntertwining} the multiplication maps are intertwining
\[
M_\Phi\in \Hom_H(\pi_{\xi,\lambda},\pi_{\xi-e_1,\lambda-e_1}) 
,\qquad M_\Psi\in \Hom_H(\pi_{\xi,\lambda},\pi_{\xi-\hat{e}_{n+1},\lambda-\hat{e}_{n+1}}\otimes \det|_H),
\]
where $\hat{e}_j=\sum_{k\neq j}e_k$. The following proposition tells us that $\Phi$ and $\Psi$ together with the determinant generates all polynomials that satisfies these two properties.
\begin{proposition}
The only non-zero polynomials that satisfies properties \emph{(P1)} and \emph{(P2)} are, up to a constant, of the form
\[
\Phi^i\Psi^j \operatorname{det}^k\qquad (i,j,k\in \NN).
\]
\end{proposition}
\begin{proof}
First we note that all characters of $H$ are of the form $\chi=\det^r$ for some integer $r$. Consider $q$ a non-zero polynomial satisfying \emph{(P1)} and \emph{(P2)} then using the $H$-equvariance from the left we get
\[
q\begin{pmatrix}
    \overline{n}_H &\\
    y^\intercal & 1
\end{pmatrix}=\det(\overline{n}_H)^r q\begin{pmatrix}
    I_n & \\
    y^\intercal & 1
\end{pmatrix}=q\begin{pmatrix}
    I_n & \\
    y^\intercal & 1
\end{pmatrix},\qquad (\overline{n}_H\in \overline{N}_H, \,y \in \RR^n).
\]
Furthermore, using the $N_H$-invariance from both sides we have
\[
q\begin{pmatrix}
    I_n & \\
    y^\intercal & 1
\end{pmatrix}=q\bigg(\begin{pmatrix}
    n_H^{-1} & \\
    & 1
\end{pmatrix}\begin{pmatrix}
    I_n & \\
    y^\intercal & 1
\end{pmatrix}\begin{pmatrix}
    n_H & \\
    & 1
\end{pmatrix}\bigg)=q\begin{pmatrix}
    I_n & \\
    y^\intercal n_H & 1
\end{pmatrix},\qquad (n_H\in N_H,\, y\in \RR^n),
\]
which implies that $q$ does not depend on all of $y$ but only on $y_1$. We also see, using the $M_HA_H$-equivariance from both sides, that 
\[
|\ell_1|^{r-\mu_1}_{r+\epsilon_1}\cdots |\ell_n|^{r-\mu_n}_{r+\varepsilon_n}q\begin{pmatrix}
    I_n &\\
    y_1e_1^T & 1
\end{pmatrix}=q\bigg(\begin{pmatrix}
    \ell & \\
    & 1
\end{pmatrix}\begin{pmatrix}
    I_n &\\
    y_1e_1^\intercal & 1
\end{pmatrix}\begin{pmatrix}
    \ell^{-1} & \\
    & 1
\end{pmatrix}\bigg)=q\begin{pmatrix}
    I_n &\\
    \ell_1^{-1}y_1e_1^\intercal & 1
\end{pmatrix},
\]
for all $\ell=\diag(\ell_1,\dots,\ell_n)\in M_HA_H$, which implies $\mu_2=\cdots=\mu_n=r$ and $\varepsilon_2=\cdots=\varepsilon_n=[r]$  but also that 
\[
q\begin{pmatrix}
    I_n &\\
    y_1e_1^\intercal & 1
\end{pmatrix}=\operatorname{const}\times |y_1|^{r-\mu_1}_{\varepsilon_1+[r]},
\]
as it is homogeneous in $y_1$. On the dense subset $\overline{N}_GP_G$ of $G$ we then see that $q$ coincides with $|\Phi|^{r-\mu_1}_{[r]+\varepsilon_1}\allowbreak|\Psi|^{r-\mu_{n+1}}_{[r]+\varepsilon_{n+1}}\allowbreak|\det|^{\mu_{n+1}}_{\varepsilon_{n+1}}$ so by continuity they coincide on all of $G$. By requiring that this is a polynomial the result follows.
 \end{proof}
\subsection{The source operators}\label{subsec:SourceOperator}
We now wish to essentially "conjugate" the multiplication operators $M_\Phi$ and $M_\Psi$ by Knapp--Stein intertwining operators in the following way (we omit $\xi$ for simplicity)
\[\begin{tikzcd}
	{\pi_{\lambda}} & {\pi_{\lambda-w^{-1}e_1}} \\
	{\pi_{w\lambda}} & {\pi_{w\lambda-e_1} }
	\arrow["{T^w_{\lambda}}"', from=1-1, to=2-1]
	\arrow["{M_\Phi}"', from=2-1, to=2-2]
	\arrow["{T^{w^{-1}}_{w\lambda-e_1}}"', from=2-2, to=1-2]
	\arrow[dashed, from=1-1, to=1-2]
\end{tikzcd}\]
All the maps in this diagram are intertwining the $H$-action so the composition will also intertwine the $H$-action.
By using the transposition $w=(1\;i)$ in this diagram we can create maps $\pi_\lambda\to\pi_{\lambda-e_i}$ and by replacing $M_\Phi$ by $M_\Psi$ we can get maps $\pi_\lambda\to \pi_{\lambda-\hat{e}_i}$. 
\begin{remark}
Applying $M_{\det}$ does not yield any meaningful results, as it is equivalent to tensoring with the determinant, which simply shifts all the parameters $\lambda$ by one. Moreover, “conjugating” by the Knapp–Stein operators has no effect, since $M_{\det}$ and $T^w_{\xi, \lambda}$ commute. Therefore, we disregard the multiplication map by the determinant.
\end{remark}

Let $w_i$ denote the simple transposition $(i\;i+1)$ then let $\Phi_1=\Phi$ and $\Psi_{n+1}=\Psi$ then define
\[
\Phi_{i+1}(g)=\Phi_i(gw_i),\qquad \Psi_{i}(g)=\Psi_{i+1}(gw_i).
\]
Let $\lambda_{i,j}=\lambda_i-\lambda_j-1$ and consider the right regular action $r$ of $G$ on $C^\infty(G)$ which has infinitesimal action on the standard basis $E_{i,j}$ given by the differential operators
\begin{equation}\label{eq:DefVarepsilon}
\varepsilon^{i,j}:=dr(E_{i,j})f(g)= \sum_{k=1}^{n+1}g_{ki}\partial_{kj}f(g).
\end{equation}
Let $\tilde{\varepsilon}^{i,j}=(-1)^{i+j+1}\varepsilon^{i,j}$ then we have the following result
\begin{theorem}
Let $\calD_1=M_\Phi$ and $\calF_{n+1}=M_\Psi$. The following $H$-intertwining operators are differential for generic $(\xi,\lambda)$: 
\begin{align*}
\calD_{i+1}(\lambda)&=\lambda_{i+1,i}T_{w_i(\xi,\lambda)+(e_i,e_i)}^{w_i}\circ \calD_{i}(w_i\lambda)\circ \Big(T^{w_i}_{w_i(\xi,\lambda)}\Big)^{-1},\qquad &(i=1,\dots,n),\\
\calF_i(\lambda)&=\lambda_{i+1,i}T_{w_i(\xi,\lambda)+(e_{i+1},e_{i+1})}^{w_i}\circ \calF_{i+1}(w_i\lambda)\circ \Big(T^{w_i}_{w_i(\xi,\lambda)}\Big)^{-1},\qquad &(i=1,\dots,n),
\end{align*}
and explicitly given by 
\begin{equation}\label{eq:calD}
\calD_{i+1}(\lambda)=\det\begin{pmatrix}
    \Phi_1 & \lambda_{i+1,1} &0 & \cdots & 0\\
    \Phi_2 & \varepsilon^{2,1} & \lambda_{i+1,2} & \cdots &0\\
    \vdots & \vdots & \vdots & \ddots & \vdots \\
    \Phi_{i}& \varepsilon^{i,1} & \varepsilon^{i,2}& \cdots &\lambda_{i+1,i}\\
    \Phi_{i+1} & \varepsilon^{i+1,1} & \varepsilon^{i+1,2}&\cdots & \varepsilon^{i+1,i}
\end{pmatrix},
\end{equation}
\begin{equation}\label{eq:calF}
\calF_i(\lambda)=\det\begin{pmatrix}
    \Psi_{n+1} & \lambda_{n+1,i} & 0 &\cdots &0\\
    \Psi_n & \tilde{\varepsilon}^{n+1,n} & \lambda_{n,i} & \cdots  & 0\\
    \vdots & \vdots & \vdots & \ddots & \vdots\\
    \Psi_{i+1} & \tilde{\varepsilon}^{n+1,i+1} & \tilde{\varepsilon}^{n,i+1} &\cdots & \lambda_{i+1,i}\\
    \Psi_i & \tilde{\varepsilon}^{n+1,i} & \tilde{\varepsilon}^{n,i}& \cdots & \tilde{\varepsilon}^{i+1,i}
\end{pmatrix},
\end{equation}
where these determinants with non-commuting entries should be expanded such that in each term the factors comes in the same order as the columns. These operators then intertwine the following spaces: 
\[
\calD_{i}(\lambda)\in \operatorname{Hom}_H(\pi_{\xi,\lambda},\pi_{\xi-e_{i},\lambda-e_{i}})\qquad \text{and}\qquad \calF_i(\lambda)\in \operatorname{Hom}_H(\pi_{\xi,\lambda},\pi_{\xi-\hat{e}_i,\lambda-\hat{e}_i}\otimes \det|_H).
\]

\end{theorem}
\begin{proof}
The proof for these formulas can be found in \cite[Proposition 5.2]{DitlevsenFrahm24}.
\end{proof}

We the context is clear we write $\calD_i$ and $\calF_j$ instead of $\calD_i(\lambda)$ and $\calF_j(\lambda)$. With this in mind we define $\calD_i^\alpha $ with $\alpha\in \NN$ as follows
\[
\calD_i^\alpha=\calD_i(\lambda -(\alpha-1) e_i)\circ \cdots \circ\calD_i(\lambda),
\]
and similarly for $\calF_j^\alpha$.
\begin{remark}
We notice that the definition using conjugation by Knapp--Stein intertwining operators is only valid for the $(\xi,\lambda)$ where the inverse of the Knapp--Stein operators exists, which they do for generic $(\xi,\lambda)$. However, the differential operator expressions are valid for any choice $\lambda$ and are also independent of $\xi$.
\end{remark}
In the following we let $k\in \{0,\dots,n\}$ denote the index of the embedding $p_k$ used to define $p_k$-differential symmetry breaking operators.
\begin{corollary}\label{cor:restVanish}
We have the following relations
\[
\rest_k\circ \calD_{k+1}=\lambda_{k+1,1}\cdots\lambda_{k+1,k},\qquad \rest_k\circ\calD_i=0,\qquad (i=1,\dots,k)
\]
and 
\[
\rest_k\circ \calF_{k+1}=\lambda_{n+1,k+1}\cdots\lambda_{k+2,k+1}\det|_H,\qquad \rest_k\circ \calF_j=0\qquad(j=k+2,\dots,n+1).
\]
\end{corollary}
\begin{proof}
It is easy to verify that 
\[
\rest_k(\Phi_i)(h)=\Phi_i(hx_k)=\delta_{i+1,k},\qquad \rest_k(\Psi_i)(h)= \Psi_i(hx_k)=\delta_{i+1,k}\det(h),
\]
where $\delta_{i,j}$ is Kronecker's delta.
\end{proof}
In view of Corollary \ref{cor:restVanish} for $\alpha\in \NN^n$  we consider the following source operators
\begin{equation}\label{eq:DefinitionL}
\calL_{\alpha,k}=\calF_1^{\alpha_1}\circ\cdots\circ \calF_k^{\alpha_k}\circ\calD_{k+2}^{\alpha_{k+1}}\circ\cdots\circ \calD_{n+1}^{\alpha_n},
\end{equation}
which are $H$-intertwining maps in
\[
 \Hom_H\left(\pi_{\xi,\lambda},\pi_{(\xi,\lambda)-\sum_{i=1}^p\alpha_i(\hat{e}_i,\hat{e}_i)-\sum_{i=k+1}^n\alpha_i(e_{i+1},e_{i+1})}\otimes \operatorname{det}^{\sum_{i=1}^k\alpha_i}|_H\right).
\]
If we compose with $\rest_k$ we end up in the representation $\tau_{\eta_{\alpha,k}(\xi),\nu_{\alpha,k}(\lambda)}$ where 
\[
\eta_{\alpha,k}(\xi)=(\xi_1+\alpha_1,\dots,\xi_k+\alpha_k,\xi_{k+2}+\alpha_{k+1},\dots,\xi_{n+1}+\alpha_n)=\eta_{k}(\xi)+[\alpha],
\]
and 
\begin{align*}
\nu_{\alpha,k}(\lambda)&=(\lambda_1+\tfrac{1}{2}+\alpha_1,\dots,\lambda_k+\tfrac{1}{2}+\alpha_k,\lambda_{k+2}-\tfrac{1}{2}-\alpha_{k+1},\dots,\lambda_{n+1}-\tfrac{1}{2}-\alpha_{n})\\
&=\nu_{k}(\lambda)+\sum_{i=1}^k\alpha_ie_i-\sum_{j=k+1}^n\alpha_je_j,
\end{align*}
where $\eta_k$ and $\nu_k$ are defined in Example \ref{ex:restrictionMap}.
This composition then defines a $p_k$-differential symmetry breaking operator
\[
\rest_k\circ\calL_{\alpha,k}\in \operatorname{DSBO}_k(\pi_{\xi,\lambda},\tau_{\eta_{\alpha,k}(\xi),\nu_{\alpha,k}(\lambda)}).
\]
This then raises two natural questions 
\begin{enumerate}
    \item For which parameters $\alpha$ and $(\xi, \lambda)$ does the composition $\rest_k \circ \calL_{\alpha,k}$ become zero?
    \item When the composition is non-zero, to what extent do the operators $\rest_k \circ \calL_{\alpha,k}$ fully capture all $p_k$-differential operators?
\end{enumerate}
For generic $(\xi,\lambda)$ the two questions are answered by the following Theorem and for non-generic $(\xi,\lambda)$ we give a full answer for $n=2$ in the last section. 
\begin{theorem} \label{thm:genericDSBO}
Let $k\in \{0,\dots,n\}$ and 
\[
\beta_\ell(\lambda,\nu)=\sum_{i=1}^\ell(\nu_i-\lambda_i-\tfrac{1}{2}),\qquad \beta_\ell'(\lambda,\nu)=\sum_{i=n+1-\ell}^n(\lambda_{i+1}-\nu_i-\tfrac{1}{2}).
\]
Consider the following subset of parameters in $(\ZZ/2\ZZ)^{n+1}\times \CC^{n+1}\times (\ZZ/2\ZZ)^n\times \CC^n$:
\begin{align*}
L_k&=\bigcap_{\ell=1}^k\bigg\{(\xi,\lambda,\eta,\nu)\,\Big|\,\beta_\ell(\lambda,\nu)\in \NN\quad\text{and}\quad \sum_{i=1}^k\eta_i+\xi_i=[\beta_\ell(\lambda,\nu)]\bigg\}\\
&\cap \bigcap_{\ell=1}^{n-k}\Big\{(\xi,\lambda,\eta,\nu)\,\Big|\,\beta_\ell'(\lambda,\nu)\in \NN\quad\text{and}\quad\sum_{i=n+1-\ell}^n\eta_i+\xi_{i+1}=[\beta_\ell'(\lambda,\nu)]\Big\}.
\end{align*}
If $(\xi,\lambda,\eta,\nu)\notin L_k$ then 
\[
\operatorname{DSBO}_k(\pi_{\xi,\lambda},\tau_{\eta,\nu})=\{0\}.
\]
For $(\xi,\lambda,\eta,\nu)\in L_k\cap \{(\xi,\lambda,\eta,\nu)\,|\,\lambda_i-\lambda_j\notin\ZZ\}$ we have that
\[
\operatorname{Hom}_H(\pi_{\xi,\lambda},\tau_{\eta,\nu})=\operatorname{DSBO}_k(\pi_{\xi,\lambda},\tau_{\eta,\nu})=\CC (\rest_k\circ \calL_{\alpha,k}),
\]
with $\alpha_i=\beta_i(\lambda,\nu)-\beta_{i-1}(\lambda,\nu)$ for $1\leq i\leq k$ and $\alpha_i=\beta_{n+1-i}'(\lambda,\nu)-\beta_{n-i}'(\lambda,\nu)$ for $k+1\leq i\leq n$.
\end{theorem}
We spend the rest of this section proving this theorem, except showing that $\rest_k\circ\calL_{\alpha,k}$ is non-zero which we postpone for the next section.

If there exists a $p_k$-differential symmetry breaking operator then it must have a corresponding distributional kernel $K$ that has support on $x_kP_G$ in $G/P_G$. Let 
$K_k(g)=K(x_kg)$ which has support at the origin $eP_G$ and consider it on the open dense set $\overline{N}_GP_G$ in $G/P_G$ which is isomorphic to $\R^{n(n+1)/2}$ and since it has support at the origin we get by classical distribution theory that
\begin{equation}\label{eq:distkernel}
K_k(\overline{n})=\sum_m a(m) \delta^{(m)}(\overline{n}),\qquad(a(m)\in \CC),
\end{equation}
where $m$ are lower triangular matrices with entries from $\NN$ and 
\[
\delta^{(m)}(\overline{n})=\prod_{1\leq j<i\leq n+1} \delta^{(m_{ij})}(\overline{n}_{ij}).
\]
Due to the equivariance property \eqref{eq:EquivarianceKernel}, we get the following equivariance property for $K_k$
\begin{equation}\label{eq:EquivarianceKernelShift}
K_k(x_k^{-1}m_Ha_Hn_Hx_k\overline{n}m_Ga_Gn_G)= \xi(m_G)\eta(m_H)a_G^{\lambda-\rho_G}a_H^{\nu+\rho_H}K_k(\overline{n}).
\end{equation}

\begin{lemma}\label{lemm:DifferentialEquationGeneral}
For $\ell\in\{1,\dots,k\} $ the distribution kernels $K_k$ satisfy the following differential equations
\begin{equation}\label{eq:Homogenous diff equation}
\sum_{i=\ell+1}^{n+1}\sum_{j=1}^\ell\overline{n}_{ij}\partial_{ij}=\ell(\ell-n)+\sum_{i=1}^\ell (\lambda_i-\nu_i-\tfrac{1}{2}),
\end{equation}
and for $\ell\in \{1,\dots,n-k\}$ they also satisfy
\begin{equation}\label{eq:Homogenous diff equation 2}
\sum_{i=n+2-\ell}^{n+1}\sum_{j=1}^{n+1-\ell}\overline{n}_{ij}\partial_{ij}=\ell(\ell-n)+\sum_{i=n+1-\ell}^n(\nu_i-\lambda_{i+1}-\tfrac{1}{2}).
\end{equation}
Similarly, for $\ell\in \{1,\dots,k\}$
\[
K_k\Big(\gamma_\ell(-1)\overline{n}\gamma_\ell(-1)\Big)=(-1)^{\sum_{i=1}^\ell\xi_i+\eta_i}K_k(\overline{n}),
\]
and for $\ell \in \{1,\dots,n-k\}$
\[
K_k\Big(\delta_\ell(-1)\overline{n}\delta_\ell(-1)\Big)=(-1)^{\sum_{i=n+1-\ell}^n\eta_i+\xi_{i+1}}K_k(\overline{n}).
\]
\end{lemma}
\begin{proof} 
Consider the matrices in $M_GA_G$
\[
\gamma_\ell(t)=\diag(t,\dots,t,1,\dots,1),\qquad \delta_\ell(t)=\diag(1,\dots,1,t,\dots,t),\qquad(t\in \RR^\times),
\]
where there are exactly $\ell$ entries containing $t$'s.
Here we note that $x_k\gamma_\ell(t)x_k^{-1}=\gamma_\ell(t)$ and $x_k\delta_\ell(t)x_k^{-1}=\diag(1,\dots,1,t,\dots,t,1)$ for $\ell\leq k$. Thus
\[
K_k(\gamma_\ell(t)^{-1}\overline{n}\gamma_\ell(t))=\gamma_\ell(t)^{-\nu-\rho_H}\gamma_\ell(t)^{\lambda-\rho_G} K_k(\overline{n}),
\]
and 
\[
K_k(\delta_\ell(t)\overline{n}\delta_\ell(t)^{-1})=(x_k\delta_\ell(t)x_k^{-1})^{\nu+\rho_H}\delta_\ell(t)^{-\lambda+\rho_G} K_k(\overline{n}).
\]
Now, taking $\frac{d}{dt}|_{t=1}$ we get the first two differential equations. An almost completely identical argument with $t=-1$ without taking derivatives gives the last two equations.\end{proof}

Now, applying these differential equations to \eqref{eq:distkernel}, using that $\delta^{(m_{ij})}(\overline{n}_{ij})$ is homogeneous of degree $-m_{ij}-1$, we get that $\beta_k(\lambda,\nu),\beta_k'(\lambda,\nu)\in \NN$ for $K_k$ to exist. Using the last two equations we get the condition on $(\xi,\eta)$ in $L_k$. This proves the first statement of Theorem \ref{thm:genericDSBO}.

For the second statement we use that $(G,H)=(\GL_{n+1}(\RR), \GL_n(\RR))$ is a so-called multiplicity one pair, meaning that if $\pi$ and $\tau$ are irreducible smooth admissible Fréchet representations of $G$ (respectively of $H$), then $\dim\Hom_H(\pi|_H,\tau)\leq 1$. This was shown by Sun--Zhu in \cite{SunZhu2012}. Finally, it is clear that if $(\xi,\lambda)$ is generic and $(\xi,\lambda,\eta,\nu)\in L_k$ then $(\eta,\nu)$ is also generic and so $\tau_{\eta,\nu}$ is irreducible and thus by the multiplicity one property we get $\dim\operatorname{Hom}_H(\pi_{\xi,\lambda},\tau_{\eta,\nu})\leq 1$. A more algebraic approach is possible using the duality Theorem from \cite{KobayashiPevzner16_1} which gives an isomorphism between differential symmetry breaking operators and morphisms between Verma modules. Then, for generic parameters a multiplicity free formula for the branching of Verma module has been proved in \cite{Kobayashi12Verma}. Assuming $\rest_k\circ \calL_{\alpha,k}\neq 0$ we get the last result of Theorem \ref{thm:genericDSBO}.

\section{Differential symmetry breaking operators as residues}
In this section we introduce a meromorphic family of symmetry breaking operators from \cite{DitlevsenFrahm24} and show that the differential symmetry breaking operators obtained from the source operator method show up as residues of this family. Using this result, we show that $\rest_k\circ \calL_{\alpha,k}$ is non-zero for generic $(\xi,\lambda)$.

\subsection{Recollection of the holomorphic family} 
In this subsection we simply restate some definitions and results from \cite{DitlevsenFrahm24} that will be used later. Consider the the function on $G$ given by 
\[
K_{\xi,\lambda}^{\eta,\nu}(g)=\prod_{k=1}^{n+1}|\kappa_k(g)|^{s_k}_{\delta_k}\prod_{\ell=1}^n|\theta_\ell(g)|^{t_\ell}_{\varepsilon_\ell},
\]
where 
\[
\kappa_k(g)=\det(g_{ij})_{\substack{i=n+2-k,\dots,n+1, j=1,\dots,k}},\qquad \theta_\ell(g)=\det(g_{ij})_{i=n+1-\ell,\dots,n,j=1,\dots,\ell},
\]
and for $i=1,\dots,n$ we have
\[
s_i=\lambda_i-\nu_{n+1-i}-\tfrac{1}{2},\quad t_i=\nu_{n+1-i}-\lambda_{i+1}-\tfrac{1}{2},\quad\delta_i=\xi_i+\eta_{n+1-i},\quad \varepsilon_i=\eta_{n+1-i}+\xi_{i+1},
\]
with $s_{n+1}=\lambda_{n+1}+\frac{n}{2}$ and $\delta_{n+1}=\xi_{n+1}$. The parameters $(\delta,s,\varepsilon,t)$ are sometimes referred to as \emph{spectral parameters} and, by abusing notation, we use them interchangeably with the induction parameters $(\xi,\lambda,\eta,\nu)$ i.e. $K_{\delta,s}^{\varepsilon,t}=K_{\xi,\lambda}^{\eta,\nu}$. Note that $\Phi=\kappa_1$ and $\Psi=\theta_n$. The function $K_{\xi,\lambda}^{\eta,\nu}$ is in $\calD'(G)_{\xi,\lambda}^{\eta,\nu}$ whenever its locally integrable which it is when $\Re(\lambda_i-\nu_j)>-\frac{1}{2}$ for $i+j\leq n+1$ and $\Re(\nu_j-\lambda_i)>-\frac{1}{2}$ when $i+j\geq n+2$. Let 
\[
\gamma(\xi,\lambda,\eta,\nu)=\prod_{i+1\leq n+1}\Gamma\left(\tfrac{\lambda_i-\nu_j+\frac{1}{2}+[\xi_i+\eta_{j}]}{2}\right)\prod_{i+j\geq n+2}\Gamma\left(\tfrac{\nu_j-\lambda_i+\frac{1}{2}+[\eta_j+\xi_i]}{2}\right),
\]
then $\mathbf{K}_{\xi,\lambda}^{\eta,\nu}:=\gamma(\xi,\lambda,\eta,\nu)^{-1}K_{\xi,\lambda}^{\eta,\nu}$ is holomorphic and defines a holomorphic family of symmetry breaking operators by \eqref{eq:SBOinTermsOfKernel} denoted by $\mathbf{A}_{\xi,\lambda}^{\eta,\nu}$. We then have the following functional equations with the Knapp--Stein intertwining operators 
\begin{equation}\label{eq:functionalEq1}
\mathbf{T}_{\eta,\nu}^{w_i}\circ \mathbf{A}_{\xi,\lambda}^{\eta,\nu}=\frac{\alpha(\xi,\eta)}{\Gamma(\frac{\nu_{i+1}-\nu_i+1+[\eta_{i+1}+\eta_i]}{2})}\mathbf{A}_{\xi,\lambda}^{w_i(\eta,\nu)},
\end{equation}
and 
\begin{equation}\label{eq:functionalEq2}
\mathbf{A}_{\xi,\lambda}^{\eta,\nu}\circ \mathbf{T}_{w_i(\xi,\lambda)}^{w_i}=\frac{\alpha'(\xi,\eta)}{\Gamma(\frac{\lambda_i-\lambda_{i+1}+1+[\xi_i+\xi_{i+1}]}{2})}\mathbf{A}_{w_i(\xi,\lambda)}^{\eta,\nu},
\end{equation}
where $\alpha(\xi,\eta)$ and $\alpha'(\xi,\eta)$ are non-zero scalars. Here we abuse notation and use the symbol $T$ for the Knapp--Stein intertwining operators for both $\pi_{\xi,\lambda}$ and $\tau_{\eta,\nu}$ but the indices indicates which one we are considering. The distributional kernel also satisfies the two following Bernstein--Sato identities 
\begin{equation}\label{eq:BernsteinsatoF}
\calF_i(-\lambda)\mathbf{K}_{\xi,\lambda}^{\eta,\nu}=p_i(\xi,\lambda,\eta,\nu)\mathbf{K}_{\xi+\widehat{e}_i,\lambda+\widehat{e}_i}^{\eta+\mathds{1},\nu+\mathds{1}},
\end{equation}
and
\begin{equation}\label{eq:BernsteinsatoD}
\calD_i(-\lambda)\mathbf{K}_{\xi,\lambda}^{\eta,\nu}=q_i(\xi,\lambda,\eta,\nu)\mathbf{K}_{\xi+e_{i},\lambda+e_{i}}^{\eta,\nu},
\end{equation}
where
\[
p_i(\xi,\lambda,\eta,\nu)=\varepsilon_i(\xi,\eta)\prod_{j=1}^n\Big(\frac{\nu_j-\lambda_i+\tfrac{1}{2}}{2}\Big)^{1-[\xi_i+\eta_j]}=\varepsilon_i(\xi,\eta)\prod_{j=1}^n\frac{\Gamma(\frac{\nu_j-\lambda_i+\frac{3}{2}+[\xi_i+\eta_j+1]}{2})}{\Gamma(\frac{\nu_j-\lambda_i+\frac{1}{2}+[\xi_i+\eta_j]}{2})},
\]
and 
\[
q_i(\xi,\lambda,\eta,\nu)=\varepsilon_i'(\xi,\eta)\prod_{j=1}^n\Big(\frac{\lambda_{i}-\nu_j+\tfrac{1}{2}}{2}\Big)^{1-[\xi_{i}+\eta_j]}=\varepsilon_i'(\xi,\eta)\prod_{j=1}^n\frac{\Gamma(\frac{\lambda_i-\nu_j+\frac{3}{2}+[\xi_i+\eta_j+1]}{2})}{\Gamma(\frac{\lambda_i-\nu_j+\frac{1}{2}+[\xi_i+\eta_j]}{2})},
\]
with $\varepsilon_i,\varepsilon_i'$ being non-zero scalars.

\subsection{Bernstein--Sato identities}
In this subsection we establish a functional identity relating $\mathbf{A}_{\xi,\lambda}^{\eta,\nu}\circ \calL_{\alpha,k}$ to another $\mathbf{A}_{\xi',\lambda'}^{\eta',\nu'}$. This is done by using the Bernstein--Sato identities $\eqref{eq:BernsteinsatoF}$ and \eqref{eq:BernsteinsatoD}.
These identities are on the level of the distributional kernel $\mathbf{K}_{\xi,\lambda}^{\eta,\nu}$ and as 
\[
\mathbf{A}_{\xi,\lambda}^{\eta,\nu}f(h)=\langle\mathbf{K}_{\xi,\lambda}^{\eta,\nu},f(h\,\cdot)\rangle_{G/P_G},
\]
we need to find the transpose of $\calF_i$ and $\calD_j$ with respect to $\langle \, \cdot \,,\, \cdot\,\rangle_{G/P_G}$ in order to say something about $\mathbf{A}_{\xi,\lambda}^{\eta,\nu}\circ \calD_i$ and $\mathbf{A}_{\xi,\lambda}^{\eta,\nu}\circ \calF_j$. 

\begin{lemma}
The transpose of $\calD_i(\lambda)$ and $\calF_i(\lambda)$ with respect to $\langle \, \cdot \, , \, \cdot \,\rangle_{G/P_G}$ are given by
\[
\calD_i(\lambda)^\intercal=\calD_i(-\lambda+e_{i}),\qquad \text{and}\qquad \calF_i(\lambda)^\intercal=\calF_i(-\lambda+\hat{e}_i).
\]
\end{lemma}
\begin{proof}
We show that this is the case for $\calD_2(\lambda)$ then the rest of the $\calD_i(\lambda)$ will follow similarly using the recursive definition. Writing 
\[
\calD_2(\lambda)=\lambda_{2,1}T^{w_1}_{(\xi,\lambda)-(e_2,e_2)}\circ M_\Phi\circ \big(T^{w_1}_{w_1(\xi,\lambda)}\big)^{-1},
\]
and using the fact that $(T_{\xi,\lambda}^w)^\intercal=T_{w(\xi,-\lambda)}^w$ and $M_{\Phi_1}^\intercal=M_{\Phi_1}$ and writing the inverse of the Knapp--Stein intertwiner using $c_1(\xi,\lambda)$ from \eqref{eq:InverseTony} we get that 
\[
\langle\calD_2(\lambda)\phi,\psi \rangle_{G/P_G} =\frac{\lambda_{2,1}c_1(\xi+e_2,\lambda-e_2)}{(\lambda_{1,2}+1)c_1(\xi,\lambda)}\langle \phi,\calD_2(-\lambda+e_2)\psi\rangle_{G/P_G}.
\]
The scalar in front can then be verified to be one using \eqref{eq:c-function}.
\end{proof}
Using this result we can rewrite the Bernstein--Sato identities \eqref{eq:BernsteinsatoF} and \eqref{eq:BernsteinsatoD} in the following way
\[
\mathbf{A}_{\xi,\lambda}^{\eta,\nu}\circ \calF_i(\lambda+\hat{e}_i)=p_i(\xi,\lambda,\eta,\nu)\mathbf{A}_{(\xi,\lambda)+(\hat{e}_i,\hat{e}_i)}^{(\eta,\nu)+(\mathds{1},\mathds{1})},
\]
and
\[
\mathbf{A}_{\xi,\lambda}^{\eta,\nu}\circ\calD_i(\lambda+e_{i})=q_i(\xi,\lambda,\eta,\nu)\mathbf{A}_{(\xi,\lambda)+(e_{i},e_{i})}^{\eta,\nu}.
\]
In the first identity it appears as if we equate two operators that maps to two different spaces, namely $\tau_{\eta,\nu}$ and $\tau_{(\eta,\nu)+(\mathds{1},\mathds{1})}$, however we can note that $\tau_{\eta,\nu}\otimes \det_H\simeq \tau_{(\eta,\nu)+(\mathds{1},\mathds{1})}$ so the spaces are the same. Iterating these identities we get the following result.
\begin{proposition}\label{Prop:IterratedBSIdentities}
For $\alpha\in \NN$ we have the following identities 
\[
\mathbf{A}_{\xi,\lambda}^{\eta,\nu}\circ \calF_i(\lambda+\alpha\hat{e}_i)^\alpha=p_i^{(\alpha)}(\xi,\lambda,\eta,\nu)\mathbf{A}_{(\xi,\lambda)+\alpha(\hat{e}_i,\hat{e}_i)}^{(\eta,\nu)+\alpha(\mathds{1},\mathds{1})},
\]
and
\[
\mathbf{A}_{\xi,\lambda}^{\eta,\nu}\circ \calD_i(\lambda+\alpha e_i)^\alpha=q_i^{(\alpha)}(\xi,\lambda,\eta,\nu)\mathbf{A}_{(\xi,\lambda)+\alpha(e_{i},e_{i})}^{\eta,\nu},
\]
where the polynomials $p_i^{(\alpha)}$ and $q_i^{(\alpha)}$ are given by
\[
p_i^{(\alpha)}(\xi,\lambda,\eta,\nu)=\varepsilon_i^{(\alpha)}(\xi,\eta)\prod_{j=1}^n\frac{\Gamma(\frac{\nu_j-\lambda_i+\frac{1}{2}+\alpha+[\xi_i+\eta_j+\alpha]}{2})}{\Gamma(\frac{\nu_j-\lambda_i+\frac{1}{2}+[\xi_i+\eta_j]}{2})},
\]
and
\[
q_i^{(\alpha)}(\xi,\lambda,\eta,\nu)=\varepsilon_i'^{(\alpha)}(\xi,\eta)\prod_{j=1}^n
\frac{\Gamma(\frac{\lambda_i-\nu_j+\frac{1}{2}+\alpha+[\xi_i+\eta_j+\alpha]}{2})}{\Gamma(\frac{\lambda_i-\nu_j+\frac{1}{2}+[\xi_i+\eta_j]}{2})}.
\]
Here $\varepsilon_i^{(\alpha)}$ and $\varepsilon_i'^{(\alpha)}$ are non-zero scalars.
\end{proposition}
Now note that $p_i^{(\alpha)}(\xi,\lambda,\eta,\nu)$ are not affected by the change in the parameters of $\mathbf{A}_{\xi,\lambda}^{\eta,\nu}$ caused by $\calF_j$ and $\calD_\ell$ for $\ell\neq i$ and similarly for $q_i^{(\alpha)}(\xi,\lambda,\eta,\nu)$, so when composing with $\mathbf{A}_{\xi,\lambda}^{\eta,\nu}$ with $\calF_j$ and $\calD_\ell$ the order in which we do so will not affect the final result.
If we fix $k\in \{0,\dots, n\}$ and therefore the support $x_kP_G$ of $\mathbf{K}_{\xi,\lambda}^{\eta,\nu}$, we can compose $\mathbf{A}_{\xi,\lambda}^{\eta,\nu}$ by $\calL_{\alpha,k}$ (see \eqref{eq:DefinitionL}) to get the following.

\begin{corollary}\label{cor:functionalequation}
For $\alpha\in \NN^n$ we have that 
\[
\mathbf{A}_{\xi,\lambda}^{\eta,\nu}\circ \calL_{\alpha,k}=\prod_{i=1}^k p_i^{(\alpha_i)}(\xi,\lambda,\eta,\nu)\prod_{i=k+1}^n q_{i+1}^{(\alpha_i)}(\xi,\lambda,\eta,\nu)\times \mathbf{A}_{(\xi,\lambda)+\sum_{i=1}^k\alpha_i(\hat{e}_i,\hat{e}_i)+\sum_{i=k+1}^n\alpha_i(e_{i+1},e_{i+1})}^{(\eta,\nu)+\sum_{i=1}^k\alpha_i(\mathds{1},\mathds{1})}
\]
\end{corollary}
\begin{remark}
Note this corollary also shows that the order we compose $\calF_j^{\alpha_j}$ and $\calD_\ell^{\alpha_{\ell}}$ inside $\calL_{\alpha,k}$ does not matter when composing with $\mathbf{A}_{\xi,\lambda}^{\eta,\nu}$ in the end.
\end{remark}
\subsection{Restriction as a residue}
In this subsection we prove the following Theorem.
\begin{theorem}\label{th:RestrictionAsResidue}
For $k\in \{0,\dots,n\}$ and $(\eta,\nu)=(\eta_k(\xi),\nu_k(\lambda))$ (see Example \ref{ex:restrictionMap}) we have that 
\[
\mathbf{A}_{\xi,\lambda}^{\eta,\nu}=\operatorname{const}\times e_G(\xi,\lambda)e_H(\eta,-\nu)\rest_k,
\]
where the constant is non-zero and independent of $\lambda$ and $\nu$ where
\[
e_G(\xi,\lambda)=\prod_{1\leq i<j\leq n+1}\Gamma(\tfrac{\lambda_i-\lambda_j+1+[\xi_i+\xi_j]}{2})^{-1},\qquad \text{and}\qquad e_H(\eta,\nu)=\prod_{1\leq i<j\leq n}\Gamma(\tfrac{\nu_i-\nu_j+1+[\eta_i+\eta_j]}{2})^{-1},
\]
are the Harish--Chandra $e$-functions for $G$ and $H$ respectively.
\end{theorem}
We restrict $\mathbf{A}_{\xi,\lambda}^{\eta,\nu}$ to some different parameters than $(\eta,\nu)=(\eta_k(\xi),\nu_k(\lambda))$ and then we use the Knapp--Stein intertwiners to permute the parameters to the desired ones using functional equations \eqref{eq:functionalEq1} and \eqref{eq:functionalEq2}.
\begin{lemma}\label{Lemma:KernelResidue}
For $(\delta_1,s_1),\dots,(\delta_k,s_k)=(0,-1)$ and $(\delta_{k+1},s_{k+1}),\dots,(\delta_n,s_n)=(0,0)$ together with $\Re(t_i)\geq 1$ $(i=1,\dots,n)$, the kernel $\mathbf{K}_{\delta,s}^{\varepsilon,t}$ is the distribution 
\[
\mathbf{K}_{\delta,s}^{\varepsilon,t}=\gamma'(\delta,s,\varepsilon,t)^{-1}|\kappa_{n+1}(g)|^{s_{n+1}}_{\delta_{n+1}}\prod_{j=1}^k\delta(g_{n+1,j}) |\theta_{j-1}(g)|^{t_{j-1}-1}_{\varepsilon_{j-1}} \prod_{\ell=k}^n|\theta_\ell(g)|^{t_\ell}_{\varepsilon_\ell},
\]
where
\[
\gamma'(\delta,s,\varepsilon,t)=\frac{\gamma(\delta,s,\varepsilon,t)}{\prod_{j=1}^k\Gamma(\frac{s_j+1+\delta_j}{2})}\Bigg|_{\substack{(\delta_1,s_1),\dots,(\delta_k,s_k)=(0,-1)\\
(\delta_{k+1},s_{k+1}),\dots,(\delta_n,s_n)=(0,0)}}.
\]
\end{lemma}
\begin{proof}
We restrict one parameter at a time. For $\Re(t_\ell)\geq 1$ $(\ell=1,\dots,n)$ and $\Re(s_r)\geq 0$ $(r=2,\dots,n+1)$ the functions $|\theta_\ell|^{t_\ell}_{\varepsilon_\ell}$ $(\ell=1,\dots,n)$ and $|\kappa_r|^{s_r}_{\delta_r}$ $(r=2,\dots,n+1)$ are continuous functions on $G$ which we view as an open subset of all $(n+1)\times (n+1)$ matrices over $\RR$. The residue of the Riesz distribution $|\kappa_1(g)|^{s_1}_{\delta_1}=|g_{n+1,1}|^{s_1}$ at $s_1=-1$ is given by
\[
\frac{|\kappa_1(g)|^{s_1}_{\delta_1}}{\Gamma(\frac{s_1+1+\delta_1}{2})}\Bigg |_{(\delta_1,s_1)=(0,-1)}=\delta(g_{n+1,1}),
\]
and as a functional on $C_c^\infty(G)$ extends continuously to $C(G)$. Moreover, the delta distribution satisfies $\delta(x)\varphi(x)=\delta(x)\varphi(0)$ for every $\varphi\in C(\RR)$. Therefore, we can take the residue of $|\kappa_1(g)|^{s_1}_{\delta_1}$ and multiply it with the remaining factors $|\theta_\ell|^{t_\ell}_{\varepsilon_\ell}$ and $|\kappa_r|^{s_r}_{\delta_r}$ specialized to $g_{n+1,1}=0$ giving us 
\[
\mathbf{K}_{\delta,s}^{\varepsilon,t}\Big|_{(\delta_1,s_1)=(0,-1)}=\gamma'(\delta,s,\varepsilon,t)^{-1}\delta(g_{n+1,1})|\theta_1(g)|^{t_1+s_2}_{\varepsilon_1+\delta_2}|g_{n+1,2}|^{s_2}_{\delta_2}|\theta_2(g)|^{t_2}_{\varepsilon_2}\prod_{j=3}^{n+1}|\kappa_j(g)|^{s_j}_{\delta_j}|\theta_j(g)|^{t_j}_{\varepsilon_j},
\]
where $\gamma'(\delta,s,\varepsilon,t)$ is  $\gamma(\delta,s,\varepsilon,t)$ restricted to $(\delta_1,s_1)=(0,-1)$ and without the gamma-factor used to take the residue. Next we do the exact same thing but where $|g_{n+1,2}|^{s_2}_{\delta_2}$ plays the role of $|\kappa_1(g)|^{s_1}_{\delta_1}$ and so forth. There keeps being a single Riesz distribution as a factor of the kernel since 
\[
\delta(g_{n+1,1})\cdots\delta(g_{n+1,i})|\kappa_{i+1}(g)|^{s_{i+1}}_{\delta_{i+1}}=\delta(g_{n+1,1})\cdots\delta(g_{n+1,i})|g_{n+1,i+1}|^{s_{i+1}}_{\delta_{i+1}}|\theta_i(g)|^{s_{i+1}}_{\delta_{i+1}},
\]
Using this procedure we get the claimed formula above.  
\end{proof}
We now see what this residue formula looks like on the level of the operators instead of the integral kernel. Let $w_0^H$ denote the longest Weyl group element in $W_H$ which as the permutation is given by 
\[
w_0^H=\begin{cases}
    (1\;n)(2\;n-1)\cdots (\tfrac{n}{2}\;\tfrac{n}{2}+1), & \text{if $n$ is even},\\
    (1\;n)(2\,n-1)\cdots (\tfrac{n-1}{2}\,\tfrac{n-1}{2}+2) & \text{if $n$ is odd}.
\end{cases}
\]
\begin{lemma}\label{Lemma:OperatorResidue}

Let $(\eta,\nu)=w_0^H(\eta_k(x_k^{-1}\xi),\nu_k(x_k^{-1}\lambda))$ then have that
\[
\mathbf{A}_{\xi,\lambda}^{\eta,\nu}=\gamma'(\xi,\lambda,\eta,\nu)^{-1}T_{w_0^H(\eta,\nu)}^{w_0^H}\circ \rest_k\circ T_{\xi,\lambda}^{x_k^{-1}}.
\]
\end{lemma}
\begin{proof}
First of we note that the conditions on $(\xi,\lambda,\eta,\nu)$ corresponds to $(\delta_1,s_1),\dots,(\delta_k,s_k)=(0,-1)$ and $(\delta_{k+1},s_{k+1}),\dots,(\delta_n,s_n)=(0,0)$. Let $\Re(\lambda_1)>\Re(\lambda_2)>\cdots>\Re(\lambda_{n+1})\gg 0$ and use Lemma \ref{Lemma:KernelResidue} to see that 
\begin{align*}
    \mathbf{A}_{\xi,\lambda}^{\eta,\nu}f(h)&=\int_{\overline{N}_G}\mathbf{K}_{\xi,\lambda}^{\eta,\nu}(\overline{n})f(h\overline{n})\,d\overline{n}\\
    &=\gamma'(\xi,\lambda,\eta,\nu)^{-1}\int_{\overline{N}_G}\prod_{j=1}^k\delta(\overline{n}_{n+1,j})\prod_{\ell=1}^n|\theta_\ell(\overline{n})|^{\lambda_\ell-\lambda_{\ell+1}-1}_{\xi_\ell+\xi_{\ell+1}}f(h\overline{n})\,d\overline{n}.
\end{align*}
Now we split the integration of $\overline{N}_G$ into $\overline{N}_H\times \RR^n$ by
\[
\overline{n}=\begin{pmatrix}
    \overline{n}' & \\
    y^\intercal & 1
\end{pmatrix}=\begin{pmatrix}
    \overline{n}' & \\
    & 1
\end{pmatrix}\begin{pmatrix}
    1 & \\
    y^\intercal & 1
\end{pmatrix}:=\overline{n}_H\overline{n}(y),\qquad (\overline{n}'\in \overline{N}_H, y\in \RR^n),
\]
giving us 
\[
\gamma'(\xi,\lambda,\eta,\nu)^{-1}\int_{\overline{N}_H}\int_{\RR^{n-k}}\prod_{\ell=1}^n|\theta_\ell(\overline{n}_H)|^{\lambda_\ell-\lambda_{\ell+1}-1}_{\xi_\ell+\xi_{\ell+1}}f(h\overline{n}_H\overline{n}((0,y))\,dy\,d\overline{n}_H,
\]
since $\theta_\ell$ does not depends on the coordinates in the bottom row. From \cite{TTD19}
we have the following form for the Knap--Stein intertwining operator for the longest Weyl group element 
\[
T^{w_0^H}_{w_0^H(\eta,\nu)}f(h)=\int_{\overline{N}_H}\prod_{\ell=1}^n|\theta_\ell(\overline{n}_H)|^{\lambda_\ell-\lambda_{\ell+1}-1}_{\xi_\ell+\xi_{\ell+1}} f(h\overline{n}_H)\,d\overline{n}_H.
\]
Furthermore, we can notice that $\overline{N}_G\cap x_kN_Gx_k^{-1}\simeq \RR^{n-k}$ and so
\[
T^{x_k^{-1}}_{\xi,\lambda}f(g)=\int_{\RR^{n-k}}f(gx_k^{-1}\overline{n}((0,y)))\,dy.
\]
Both of these expressions converges for the $\lambda$ we are considering. Thus, we can write $\mathbf{A}_{\xi,\lambda}^{\eta,\nu}f(h)$ as
\begin{multline*}
\gamma'(\xi,\lambda,\eta,\nu)^{-1}\int_{\overline{N}_H}\prod_{\ell=1}^n|\theta_\ell(\overline{n}_H)|^{\lambda_\ell-\lambda_{\ell+1}-1}_{\xi_\ell+\xi_{\ell+1}}\int_{\RR^{n-k}}f\big(h\overline{n}_Hx_k x_k^{-1}\overline{n}((0,y))\big)\,dy\,d\overline{n}_H\\
=\gamma'(\xi,\lambda,\eta,\nu)^{-1}\operatorname{T}_{w_0^H(\eta,\nu)}^{w_0^H}\circ \rest_k\circ \operatorname{T}_{\xi,\lambda}^{x_k^{-1}}f(h),
\end{multline*}
which establishes the result on the open set $\Re(\lambda_1)>\Re(\lambda_2)>\cdots>\Re(\lambda_{n+1})\gg 0$ and by analytic extension to all $\lambda\in \CC^{n+1}$.
\end{proof}

\begin{proof}[Proof of Theorem \ref{th:RestrictionAsResidue}]
Notice that the parameters $(x_k(\xi,\lambda),w_0^H(\eta,\nu))$  satisfy the conditions from Lemma \ref{Lemma:OperatorResidue} from which it follows that 
\[
\mathbf{A}_{x_p(\xi,\lambda)}^{w_0^H(\eta,\nu)}=\gamma'(x_k(\xi,\lambda),w_0^H(\eta,\nu))^{-1}T_{\eta,\nu}^{w_0^H}\circ\rest_k\circ T_{x_k(\xi,\lambda)}^{x_k^{-1}}.
\]
It then follows by \eqref{eq:functionalEq1} and \eqref{eq:functionalEq2} that

\begin{align*}
    \mathbf{A}_{\xi,\lambda}^{\eta,\nu}&=\prod_{1\leq i<j\leq n}\frac{\Gamma(\tfrac{\nu_i-\nu_j+1+[\eta_i+\eta_j]}{2})}{\Gamma(\tfrac{\nu_j-\nu_i+[\eta_i+\eta_j]}{2})}\prod_{i=k+2}^{n+1}\frac{\Gamma(\tfrac{\lambda_i-\lambda_{k+1}+1+[\xi_i+\xi_{k+1}]}{2})}{\Gamma(\tfrac{\lambda_{k+1}-\lambda_{i}+[\xi_i+\xi_{k+1}]}{2})}T_{w_0^H(\eta,\nu)}^{w_0^H}\circ \mathbf{A}_{x_k(\xi,\lambda)}^{w_0^H(\eta,\nu)}\circ T_{\xi,\lambda}^{x_k}\\
    &=\gamma'(x_k(\xi,\lambda),w_0^H(\eta,\nu))^{-1}\prod_{1\leq i<j\leq n}\frac{\Gamma(\tfrac{\nu_i-\nu_j+1+[\eta_i+\eta_j]}{2})}{\Gamma(\tfrac{\nu_j-\nu_i+[\eta_i+\eta_j]}{2})}\prod_{i=k+2}^{n+1}\frac{\Gamma(\tfrac{\lambda_i-\lambda_{k+1}+1+[\xi_i+\xi_{k+1}]}{2})}{\Gamma(\tfrac{\lambda_{k+1}-\lambda_{i}+[\xi_i+\xi_{k+1}]}{2})}\\
    &\times T_{w_0^H(\eta,\nu)}^{w_0^H}\circ T_{\eta,\nu}^{w_0^H}\circ\rest_k\circ T_{x_k(\xi,\lambda)}^{x_k^{-1}}\circ T_{\xi,\lambda}^{x_k}
    \\
    &=\operatorname{const}\times e_G(\xi,\lambda
    )e_H(\eta,-\nu)\rest_k
\end{align*}
where we use the Gindikin--Karpelevich formula \eqref{eq:Gindikin--Karpelevich} to get
 \[
 T_{w_0^H(\eta,\nu)}^{w_0^H}\circ T_{\eta,\nu}^{w_0^H}=\prod_{1\leq i<j\leq n}\frac{\Gamma(\frac{\nu_i-\nu_j+[\eta_i+\eta_j]}{2})\Gamma(\frac{\nu_j-\nu_i+[\eta_i+\eta_j]}{2})}{\Gamma(\frac{\nu_i-\nu_j+1+[\eta_i+\eta_j]}{2})\Gamma(\frac{\nu_j-\nu_i+1+[\eta_i+\eta_j]}{2})} \times \operatorname{id},
 \]
 and 
 \[
 T_{x_p(\xi,\lambda)}^{x_p^{-1}}\circ T_{\xi,\lambda}^{x_p}=\prod_{i=p+1}^{n+1}\frac{\Gamma(\tfrac{\lambda_i-\lambda_{p+1}+[\xi_i+\xi_{p+1}]}{2})\Gamma(\tfrac{\lambda_{p+1}-\lambda_{i}+[\xi_i+\xi_{p+1}]}{2})}{\Gamma(\tfrac{\lambda_{i}-\lambda_{p+1}+1+[\xi_i+\xi_{p+1}]}{2})\Gamma(\tfrac{\lambda_{p+1}-\lambda_{i}+1+[\xi_i+\xi_{p+1}]}{2})}\times \operatorname{id}.
 \]
 To get the scalar $\operatorname{const}\times e_G(\xi,\lambda)e_H(\eta,-\nu)$ is then a matter of bookkeeping gamma-factors.
 \end{proof}

 \subsection{Vanishing of the source operator}\label{Sec:vanishing}
Combining Theorem \ref{th:RestrictionAsResidue} and Corollary \ref{cor:functionalequation} we get that the differential symmetry breaking operators $\rest_k\circ\calL_{\alpha,k}$ are residues of the family $A_{\xi,\lambda}^{\eta,\nu}$ described in the following Theorem.
 \begin{theorem}\label{thm:DSBOasresidue}
Define $(\tilde{\xi},\tilde{\lambda})=(\xi,\lambda)\allowbreak-\sum_{i=1}^k\alpha_i(\hat{e}_i,\hat{e}_i)-\sum_{i=k+1}^n\alpha_i(e_{i+1},e_{i+1})$ and $(\tilde{\eta},\tilde{\nu})=(\eta,\nu)\allowbreak-\sum_{i=1}^k\alpha_i(\mathds{1},\mathds{1})$. If $(\tilde{\eta},\tilde{\nu})=(\eta_k(\tilde{\xi}),\nu_k(\tilde{\lambda}))$ then
\[
e_G(\tilde{\xi},\tilde{\lambda})e_H(\tilde{\eta},-\tilde{\nu})\rest_k\circ\calL_{\alpha,k}=\prod_{i=1}^k p_i^{(\alpha_i)}(\tilde{\xi},\tilde{\lambda},\tilde{\eta},\tilde{\nu})\prod_{i=k+1}^n q_{i+1}^{(\alpha_i)}(\tilde{\xi},\tilde{\lambda},\tilde{\eta},\tilde{\nu}) \mathbf{A}_{\xi,\lambda}^{\eta,\nu}.
\]
\end{theorem} 
In the remainder of this section, we will apply the differential symmetry breaking operators $ \rest_k \circ \calL_{\alpha,p} $ to a function in order to derive a sufficient condition for $\rest_k \circ \calL_{\alpha,k}$ to be non-zero. To this end, we will the above theorem to apply $\mathbf{A}_{\xi,\lambda}^{\eta,\nu}$ to the spherical function multiplied by certain polynomials, allowing us to cover the non-spherical cases. For $\xi=(0,\dots,0)$ we have the $K_G$-fixed vector $\mathbf{1}_\lambda\in \pi_{0,\lambda}$ normalized by $\mathbf{1}_\lambda(e)=1$ and similarly (by abuse of notation) we have for $\eta=(0,\dots,0)$ the $K_H$-fixed vector $\mathbf{1}_\nu\in \tau_{0,\nu}$. In \cite[Theorem 8.5]{DitlevsenFrahm24} it was then proved that 
 \begin{equation}\label{eq:sphericalvectorevaluation}
 \mathbf{A}_{0,\lambda}^{0,\nu}\mathbf{1}_\lambda=\operatorname{const}\times e_G(0,\lambda)e_H(0,-\nu)\mathbf{1}_\nu,
 \end{equation}
 where the constant is non-zero and independent of $\lambda$ and $\nu$.
Combining this result with Theorem \ref{thm:DSBOasresidue} will already give us a non-vanishing condition in the spherical case. We show a result which can be applied in the non-spherical case. We consider the following function
 \[
 \mathbf{1}_{\xi,\lambda}^{\eta,\nu}:=\Big(\kappa_{n+1}^{[\xi_{n+1}]}\prod_{i=1}^n\kappa_i^{[\xi_i+\eta_{n+1-i}]}\theta_i^{[\eta_{n+1-i}+\xi_{i+1}]}\Big)\times\mathbf{1}_{\lambda'}\in \pi_{\xi,\lambda},
 \]
where $\lambda'$ is as in the proposition below. \begin{proposition}\label{prop:Non-spherical evaluation}
We have the following 
\[
\mathbf{A}_{\xi,\lambda}^{\eta,\nu}\mathbf{1}_{\xi,\lambda}^{\eta,\nu}(e)=\operatorname{const}\times\frac{\gamma(0,\lambda',0,\nu')}{\gamma(\xi,\lambda,\eta,\nu)}e_G(0,\lambda')e_H(0,-\nu')
\]
where the non-zero constant is independent of $(\xi,\lambda,\eta,\nu)$ and if we set $\eta_0=0$ and $\xi_{n+2}=0$ then
\[
\lambda'_i=\lambda_i+\sum_{k=i}^{n+1}[\xi_{k}+\eta_{n+1-k}]+[\eta_{n+1-k}+\xi_{k+1}]
\]
and
\[
\nu_i'=\nu_i+\sum_{k=n+1-i}^{n}[\eta_{n+1-k}+\xi_{k+1}]+[\xi_{k+1}+\eta_{n-k}].
\]
\end{proposition}
\begin{proof}
Let $\Re(\lambda_1)>\Re(\nu_n)>\Re(\lambda_2)>\cdots>\Re(\lambda_{n+1})$ then the same pattern hold for $(\lambda',\nu')$ and so the operator in the following are given as an integral of continuous functions on the compact set $G/P_G\simeq K_G/M_G$. We have 
\begin{multline}
\mathbf{A}_{\xi,\lambda}^{\eta,\nu}\Big(\kappa_{n+1}^{[\xi_{n+1}]}\prod_{i=1}^n\kappa_i^{[\xi_i+\eta_{n+1-i}]}\theta_i^{[\eta_{n+1-i}+\xi_{i+1}]}\mathbf{1}_{\lambda'}\Big)(e)\\
=\int_{G/P_G}\mathbf{K}_{\xi,\lambda}^{\eta,\nu}(g)\kappa_{n+1}(g)^{[\xi_{n+1}]}\prod_{i=1}^n\kappa_i(g)^{[\xi_i+\eta_{n+1-i}]}\theta_i(g)^{[\eta_{n+1-i}+\xi_{i+1}]}\mathbf{1}_{\lambda'}(g)\,d(gP_G)
\end{multline}
Now all the factors multiplied on the spherical vector $\mathds{1}_{\lambda'}$ are factors of the kernel $\mathbf{K}_{\xi,\lambda}^{\eta,\nu}$ so taking the normalization into account we get 
\[
\frac{\gamma(0,\lambda',0,\nu')}{\gamma(\xi,\lambda,\eta,\nu)}\int_{G/P_G}\mathbf{K}_{0,\lambda'}^{0,\nu'}(g)\mathbf{1}_{\lambda'}(g)\,d(gP_G)=\operatorname{const}\times\frac{\gamma(0,\lambda',0,\nu')}{\gamma(\xi,\lambda,\eta,\nu)}e_G(0,\lambda')e_H(0,-\nu')\mathbf{1}_\nu(e)
\]
by \eqref{eq:sphericalvectorevaluation}. We note that $\gamma(0,\lambda',0,\nu')/\gamma(\xi,\lambda,\eta,\nu)$ is a polynomial by \cite[Remark 7.1]{DitlevsenFrahm24}. This is then an equality of holomorphic functions on an open set and so for the rest of the parameters it follows by the identity theorem for holomorphic functions.
\end{proof}
We can now combine Proposition \ref{prop:Non-spherical evaluation} and Theorem \ref{thm:DSBOasresidue} to get \begin{proposition}
For $(\tilde{\xi},\tilde{\lambda},\tilde{\eta},\tilde{\nu})$ as in Theorem \ref{thm:DSBOasresidue} and $(\lambda',\nu')$ as in  Proposition \ref{prop:Non-spherical evaluation} we have
\begin{multline*}
e_G(\tilde{\xi},\tilde{\lambda})e_H(\tilde{\eta},-\tilde{\nu})\rest_k\circ \calL_{\alpha,k}\mathbf{1}_{\xi,\lambda}^{\eta,\nu}(e)\\=\operatorname{const}\times\prod_{i=1}^kp_i^{(\alpha_i)}(\tilde{\xi},\tilde{\lambda},\tilde{\eta},\tilde{\nu})\prod_{i=k+1}^nq_{i+1}^{(\alpha_i)}(\tilde{\xi},\tilde{\lambda},\tilde{\eta},\tilde{\nu})\frac{\gamma(0,\lambda',0,\nu')}{\gamma(\xi,\lambda,\eta,\nu)}e_G(0,\lambda')e_H(0,-\nu'),
\end{multline*}
for a non-zero constant independent of $(\alpha,\xi,\lambda,\eta,\nu)$.
\end{proposition}
As all the functions involved only has zeros when $\lambda_i-\lambda_j\in \ZZ$ and $\tilde{\lambda}$,  $\tilde{\nu}$, $\lambda'$ and $\nu'$ are integer shifts of $(\lambda,\nu)$ the above identity gives us the following corollary.
\begin{corollary}
The operator $\rest_k\circ\calL_{\alpha,k}$ is non-zero for generic $(\xi,\lambda)$.
\end{corollary}
This Corollary then completes the proof of Theorem \ref{thm:genericDSBO}.
\section{Non generic parameters: the \texorpdfstring{$n=2$}{n=2} case}
Until now we worked in the case of generic parameters $(\xi,\lambda)$, i.e. when $\lambda_i-\lambda_j\notin \Z$ for all $i,j$. In this section we classify and construct differential symmetry breaking operators for any parameters $(\xi,\lambda)$ in the case $n=2$, that is for the pair $(\GL_3(\R),\GL_2(\R))$. In particular, we will exhibit some multiplicity two phenomenon in this situation i.e. where $\dim \operatorname{DSBO}_1(\pi_{\xi,\lambda},\tau_{\eta,\nu})=2$.

Recall that, for each element $x_k$ for $0\leq k\leq 2$, we want to find distributional kernels $K$ for which $\supp(K)=x_kP_G$. We can consider the open dense subset of $G$ given as $x_k\overline{N}_GP_G$ which contains the support of the distribution and restrict the distribution to this set. As the distribution has right $P_G$ equivariance we have to determine $K$ on $x_k\overline{N}_G$ which is isomorphic to $\RR^3$. As mentioned in Subsection \ref{subsec:SourceOperator} we consider the distributions 
\[
K_k(\overline{n})=K(x_k\overline{n}),
\]
which are all supported at the origin in $\overline{N}_G$. Writing elements of $\overline{N}_G$ in the coordinates \[\begin{pmatrix}
    1&0&0\\x&1&0\\z&y&1
\end{pmatrix},\quad (x,y,z\in \R),\] the distributional kernel can be written as a finite sum as follows
\[
K_k(\overline{n})=\sum_{i_x,i_y,i_z} a(i_x,i_y,i_z)\delta^{(i_x)}(x)\delta^{(i_y)}(y)\delta^{(i_z)}(z).
\]

We will now look at each possibility for $k=0,1,2$. We give details of the solution in the $k=1$ case, and give the results for the other two cases since the analysis is similar and less complicated.  

\subsection{Support at \texorpdfstring{$x_1P_G$}{x1PG}}
\subsubsection{Classification of differential symmetry breaking operators}
The case $k=1$ is the most involved and the only one that leads to multiplicity two. First, we find the following differential equations satisfied by $K_1$ as consequences of the equivariance property \eqref{eq:EquivarianceKernelShift}

\[
\left \{
\begin{array}{ l }
x\partial_x+z\partial_z  =\lambda_1-\nu_1-\frac{3}{2} ,\\
    y\partial_y+z\partial_z  =\nu_2-\lambda_3-\frac{3}{2} ,\\
    z(\lambda_1-\lambda_3-2-x\partial_x-z\partial_z-y\partial_y)- xy(\lambda_2-\lambda_3-1-y\partial_y) =0.
\end{array}
\right.
\]
The first two equations are consequences of Lemma \ref{lemm:DifferentialEquationGeneral}. The third one is a consequence of the equivariance property applied to the following equality:
\begin{multline*}
\begin{pmatrix}
    1&0&t\\    0&1&0\\  0&0&1
\end{pmatrix}
\begin{pmatrix}
    1&0&0\\x&1&0\\z&y&1
\end{pmatrix}
=
\begin{pmatrix}
    1&0&0\\\frac{x}{1+tz}&1&0\\\frac{z}{1+tz} &\frac{y}{1+t(z-xy)}&1
\end{pmatrix}\\\times  
\begin{pmatrix}
    1&\frac{ty(1+tz)}{1+t(z-xy)}&t(1+t(z-xy))\\0&1&-\frac{tx(1+t(z-xy))}{1+tz}\\0&0&1
\end{pmatrix}
\begin{pmatrix}
    1+tz&0&0\\0&\frac{1+t(z-xy)}{1+tz}&0\\0&0&(1+t(z-xy))^{-1}
\end{pmatrix}.
\end{multline*}
 Differentiating the equivariance relation with respect to the variable $t$ and evaluating at $t=0$ then gives the third differential equation. 

The first two equations lead to the following
\begin{align*}
    i_x+i_z&=\nu_1-\lambda_1-\frac{1}{2}\\
    i_y+i_z&=\lambda_3-\nu_2-\frac{1}{2},
\end{align*}
and so we set $n_1=\nu_1-\lambda_1-\frac{1}{2}$ and $n_2=\lambda_3-\nu_2-\frac{1}{2}$.
The parameters $n_1$ and $n_2$ correspond respectively to $\beta_1(\lambda,\nu)$ and $\beta'_1(\lambda,\nu)$ introduced in Theorem \ref{thm:genericDSBO}. Thus there can only be a non-zero operator if $n_1$, $n_2$ are non-negative integers. If so, setting $N=\min(n_1,n_2)$, we write the kernel:
\[
K_1=\sum_{j=0}^N c_j \delta^{(n_1-j)}(x)\delta^{(n_2-j)}(y)\delta^{(j)}(z). 
\]
Applying the third equation to this expression gives a recurrence relation for the coefficients $c_j$ for $0\leq j\leq N-1$ given by
\[
(\lambda_1-\lambda_3+n_1+n_2-j)(j+1)c_{j+1}=-(\lambda_2-\lambda_3+n_2-j)(n_1-j)(n_2-j)c_j. 
\]
Solving this recurrence relation gives us three cases for differential symmetry breaking operators summarized in the following proposition. 
\begin{proposition}\label{prop:DSBOk=1}
    We have the following three cases:
    \begin{enumerate}
        \item[(1)] If $n_1,\ n_2\in \N$ and there does not exist integers $k_0\leq \ell_0$ such that $\lambda_1-\lambda_3+n_1+n_2-\ell_0=\lambda_2-\lambda_3+n_2-k_0=0$, and 
        \[
        \eta_1= \xi_1+[n_1] \qquad \text{ and }\qquad  \eta_2 = \xi_3+[n_2],\] then 
        \[\dim (\DSBO_1(\pi_{\xi,\lambda},\tau_{\eta,\nu}))=1.\]
        In that situation, if $\lambda_1-\lambda_3+n_1+n_2-\ell\neq 0$ for all $\ell\leq N-1$, the distribution kernel of a basis of $\DSBO_1(\pi_{\xi,\lambda},\tau_{\eta,\nu})$ is given by
        \[
        \sum_{j=0}^N (-1)^j\frac{(-n_1)_j(-n_2)_j(\lambda_3-\lambda_2-n_2)_j}{(\lambda_3-\lambda_1-n_1-n_2)_j j!} \delta^{(n_1-j)}(x)\delta^{(n_2-j)}(y)\delta^{(j)}(z).
        \] 
        If there exists $\ell_0\leq N-1$ such that $\lambda_1-\lambda_3+n_1+n_2-\ell_0= 0$, the distribution kernel of a basis of $\DSBO_1(\pi_{\xi,\lambda},\tau_{\eta,\nu})$ is given by
        \[
        \sum_{j=\ell_0+1}^N (-1)^j\frac{(-n_1)_j(-n_2)_j(\lambda_3-\lambda_2-n_2)_j}{(\lambda_3-\lambda_1-n_1-n_2)_j j!} \delta^{(n_1-j)}(x)\delta^{(n_2-j)}(y)\delta^{(j)}(z).
        \] 
        \item[(2)] If $n_1,\ n_2\in \N$  and there exists integers $0\leq k_0\leq \ell_0\leq N-1$ such that $\lambda_1-\lambda_3+n_1+n_2-\ell_0=\lambda_2-\lambda_3+n_2-k_0=0$ and 
        \[
        \eta_1= \xi_1+[n_1] \qquad \text{ and } \qquad\eta_2= \xi_3+[n_2],
        \]
        then 
        \[
        \dim (\DSBO_1(\pi_{\xi,\lambda},\tau_{\eta,\nu}))=2.\]
        The distribution kernels of a basis of $\DSBO_1(\pi_{\xi,\lambda},\tau_{\eta,\nu})$ are given by the following 
        \[
         \sum_{j=\ell_0+1}^N (-1)^j\frac{(-n_1)_j(-n_2)_j(\lambda_3-\lambda_2-n_2)_j}{(\lambda_3-\lambda_1-n_1-n_2)_j j!} \delta^{(n_1-j)}(x)\delta^{(n_2-j)}(y)\delta^{(j)}(z),
        \] 
        and 
        \[
        \sum_{j=0}^{k_0} (-1)^j\frac{(-n_1)_j(-n_2)_j(\lambda_3-\lambda_2-n_2)_j}{(\lambda_3-\lambda_1-n_1-n_2)_j j!} \delta^{(n_1-j)}(x)\delta^{(n_2-j)}(y)\delta^{(j)}(z).
        \]
        
        \item[(3)] Otherwise 
        \[\dim (\DSBO_1(\pi_{\xi,\lambda},\tau_{\eta,\nu}))=0.\]
    \end{enumerate}
\end{proposition}

\begin{remark}
    For generic parameters, we recover the results of Theorem \ref{thm:genericDSBO} for $k=1$. Indeed, for generic parameters we are always either in case (1) or (3). Moreover, case (3) is equivalent to $(\xi,\lambda,\eta,\nu)$ not being in $L_1$. 
\end{remark}


\subsubsection{Alternative construction of the bases}
We now give an alternative constructions of the bases for $\DSBO_1(\pi_{\xi,\lambda},\tau_{\eta,\nu})$ in the first two cases using the operators $\calF_1$ and $\calD_3$. Case $(1)$ of Proposition \ref{prop:DSBOk=1} turns out to be achievable using a renormalization of the operator $\calL_{\alpha,1}$ (see Proposition \ref{prop:DSBOMultiplictyTwo} below). 

In case $(2)$ of Proposition \ref{prop:DSBOk=1} an initial idea for how to construct a differential symmetry breaking operator, is to use the operator $\rest_1\circ \calF_1^{n_1}\circ \calD^{n_2}_3$ but it is in general zero. We end up constructing the symmetry breaking operators with the use of the operators $\rest_1\circ \calF_1^n\circ \calD_3^m$ together with some differential intertwining operators for $G$ or $H$. There are explicitly given by powers of the operators $\varepsilon^{i,j}$ defined in \eqref{eq:DefVarepsilon}. 

The parameters in case $(2)$ can be expressed as: there exists $\lambda_0\in \C$ such that for $n_1,n_2\geq 1$ we have
\begin{equation}\label{eq:ParametersMultiplicityTwo}
\lambda=(\lambda_0,\lambda_0+n_1+k_0-\ell_0,\lambda_0+n_1+n_2-\ell_0),\qquad\text{and}\qquad \nu=(\lambda_0+n_1+\tfrac{1}{2},\lambda_0+n_1-\ell_0 -\tfrac{1}{2}),
\end{equation}
together with
\[
\xi=(\xi_1,\xi_2,\xi_3),\qquad\text{and}\qquad \eta=(\xi_1+n_1,\xi_3+n_2 ). 
\]
Consider the parameters 
\[
(\eta',\nu')=w_1(\eta,\nu)=((\eta_2,\eta_1),(\lambda_0+n_1-\ell_0-\tfrac{1}{2},\lambda_0+n_1+\tfrac{1}{2})),
\]
and 
\[
(\xi',\lambda')=w_2(\xi,\lambda)=((\xi_1,\xi_3,\xi_2),(\lambda_0,\lambda_0+n_1+n_2-\ell_0,\lambda_0+n_1+k_0-\ell_0)).
\]
We can then construct two types of differential symmetry breaking operators by the following diagram
\[
 \xymatrix{
 &&& \tau_{\eta',\nu'}  \ar[drrr]^{(\varepsilon^{2,1})^{\ell_0+1}}&&&\\
   \pi_{\xi,\lambda}  \ar[urrr]^{ B_0} \ar[drrr]_{(\varepsilon^{3,2})^{n_2-k_0}} &&& &&& \tau_{\eta,\nu}\\
 &&&\pi_{\xi',\lambda'} \ar[urrr]_{B_1}&&&
  }
\]
where
\[
 B_0=\rest_1\circ\calF_1^{n_1-\ell_0-1}\circ \calD_3^{n_2-\ell_0-1},
\]
and
\[
B_1=\frac{1}{(\ell_0-n_1+1)_{n_1}}\rest_1\circ \calF_1^{n_1}\circ\calD_3^{k_0}.
\]

This leads to the following result.
\begin{proposition}\label{prop:DSBOMultiplictyTwo}
    In case $(1)$ from Proposition \ref{prop:DSBOk=1} the renormalized operator 
    \[
    \frac{1}{(\lambda_1-\lambda_3+1+n_2)_{n_1}}\rest_1\circ \calF_1^{n_1}\circ \calD_3^{n_2},
    \]
    is always non-zero and a basis of $\DSBO_1(\pi_{\xi,\lambda},\tau_{\eta,\nu})$.

    In case (2) of Proposition \ref{prop:DSBOk=1} the following two operators 
    \[
    \frac{1}{(\ell_0-n_1+1)_{n_1}}\rest_1\circ \calF_1^{n_1}\circ\calD_3^{k_0}\circ \textbf{T}^{(23)}_{\xi,\lambda}
    \]
    and
    \[
     \textbf{T}^{(12)}_{\eta',\nu'}\circ  \rest_1\circ\calF_1^{n_1-\ell_0-1}\circ \calD_3^{n_2-\ell_0-1},
    \]
    are non-zero and they form a basis of $\DSBO_1(\pi_{\xi,\lambda},\tau_{\eta,\nu})$.
\end{proposition}

To prove this result we need a couple of lemmas starting with the following. 
\begin{lemma}\label{lemma:OperatorForMultiplicityTwo}
    For $\lambda=(\lambda_1,\lambda_2,\lambda_3)\in \C^3$ and $n,m\in \N$ we have 
    \begin{equation*}
        \rest_1\circ \calF_1^n\circ \calD_3^m=(\lambda_1-\lambda_3 +1+m)_n{\det}_H^n 
        \sum_{i=0}^m (n-i+1)_i a_i(m) \rest_1\circ \left(\varepsilon^{2,1}\right)^{n-i}\left(\varepsilon^{3,2}\right)^{m-i}\left(\varepsilon^{3,1}\right)^i,
    \end{equation*}
    with $a_i(m)\in \C$. The first and the last coefficients are explicitly given by 
    \begin{align*}
       a_0(m)&=(\lambda_1-\lambda_3+n+1)_m,\\ 
       a_m(m)&=(\lambda_2-\lambda_3+1)_m. 
    \end{align*}
\end{lemma}
\begin{proof}
    We first prove the formula for $m=0$. Thus we have to prove:
     \begin{equation*}
        \rest_1\circ \calF_1^n=(\lambda_1-\lambda_3 )_n{\det}_H^n \rest_1\circ \left(\varepsilon^{2,1}\right)^n.
    \end{equation*}
We have that
\[
\calF_1=\Psi_3(\varepsilon^{3,2}\varepsilon^{2,1}+\lambda_{2,1}\varepsilon^{3,1})-\lambda_{3,1}\Psi_2\varepsilon^{2,1}+\lambda_{3,1}\lambda_{2,1}\Psi_1.
\]
Together with the relations
\begin{align*}
 \varepsilon^{2,1}\Psi_1&=\varepsilon^{2,1}\Psi_3=0,\qquad &\varepsilon^{2,1}\Psi_2&=-\Psi_1,\\
 \rest_1\circ \Psi_2&={\det}_H,\qquad &\rest_1\circ \Psi_1&=\rest_1\circ \Psi_3=0, 
\end{align*}
one quickly prove the result by induction on $n$ after paying attention to the parameters of the operators.

Now for $m\neq 0$, we use the following formulas
\[
\calD_3=\Phi_1(\varepsilon^{2,1}\varepsilon^{3,2}-\lambda_{3,2}\varepsilon^{3,1})-\lambda_{3,1}\Phi_2\varepsilon^{3,2}+\lambda_{3,1}\lambda_{3,2}\Phi_3,
\]
and 
\[
\varepsilon^{a,b}\Phi_c=\delta_{b,c} \Phi_a,\qquad \rest_1 \Phi_2=1,\qquad \rest_1\Phi_1=\rest_1\Phi_2=0. 
\]
Using also the commutation relation $[\varepsilon^{3,2},\varepsilon^{2,1}]=\varepsilon^{3,1}$, this leads to 

\begin{multline*}
 \rest_1\circ \left(\varepsilon^{2,1}\right)^{n-i}\left(\varepsilon^{3,2}\right)^{m-i}\left(\varepsilon^{3,1}\right)^i\circ \calD_3\\
 =(\lambda_1-\lambda_3+n-i+1)\rest_1\circ \left(\varepsilon^{2,1}\right)^{n-i}\left(\varepsilon^{3,2}\right)^{m+1-i}\left(\varepsilon^{3,1}\right)^i\\
 +(n-i)(\lambda_2-\lambda_3+m-i+1)\rest_1\circ\left(\varepsilon^{2,1}\right)^{n-i-1}\left(\varepsilon^{3,2}\right)^{m-i}\left(\varepsilon^{3,1}\right)^{i+1}. 
\end{multline*}
The family $\rest_1\circ \left(\varepsilon^{2,1}\right)^{n-i}\left(\varepsilon^{3,2}\right)^{m-i}\left(\varepsilon^{3,1}\right)^i$ is linearly independent. Indeed, expanding the expression we can see that the term $g_{23}^{m}\rest_1\circ \partial_{13}^{n-i}\partial_{22}^{m-i}\partial_{21}^i$ only appear in the $i$-th element of the family. We introduce $\lambda'=(\lambda_1,\lambda_2,\lambda_3-1)$, which is the parameter corresponding to the image of $\calD_3$, together with the notation $a_i(m;\lambda)$ to keep track of the parameters in the proof. This leads to the following recurrence formula for $1\leq i\leq m-1$:
\begin{multline*}
  (n-i+1)_ia_i(m+1;\lambda)=(n-i+1)_i(\lambda_1-\lambda_3+n-i+1)a_i(m;\lambda')\\
  +(\lambda_2-\lambda_3+m-i)(n-i+1)(n-i+2)_{i-1}a_{i-1}(m;\lambda'),  
\end{multline*}
which shows the general form of the operator. For $i=0$ we have the recurrence formula
\[
a_0(m+1)=(\lambda_1-\lambda_3+n+1)a_0(m),
\]
with $a_0(0)=1$ which gives the result for $a_0$ after paying attention to the parameters $\lambda$ of the operators.
Finally, for the last coefficient, we have
\[
a_{m+1}(m+1;\lambda)=(\lambda_2-\lambda_3+1)a_m(m;\lambda'),
\]
with $a_0(0)=1$. This completes the proof. 
\end{proof}

The following lemma construct some differential intertwining operators between principal series representations of $G$ (respectively of $H$).
\begin{lemma}\label{lem:epsilonDiffOp}
If there exists $k_0\in \N$ such that $\lambda_3-\lambda_2=k_0$ then the operator $\left(\varepsilon^{3,2}\right)^{k_0}$ is in $\Hom_H(\pi_{\xi,\lambda},\allowbreak\pi_{\xi',\lambda'})$ with $\xi'=(\xi_1,\xi_2+[k_0],\xi_3+[k_0])$ and $\lambda'=(\lambda_1,\lambda_2+k_0,\lambda_3-k_0)$.

Similarly, if there exists $k_0\in \N$ such that $\nu_2-\nu_1=k_0$ then the operator $\left(\varepsilon^{2,1}_H\right)^{k_0}$ is in $\Hom_H(\tau_{\eta,\nu},\tau_{\eta',\nu'})$ with $\eta'=(\eta_1+[k_0],\eta_2+[k_0])$ and $\nu'=(\nu_1+k_0,\nu_2-k_0)$. Here $\varepsilon_H$ denotes the differential operators \eqref{eq:DefVarepsilon} for the subgroup $H$. 
\end{lemma}
\begin{proof}
    The proof is similar in both cases. Thus we only do it for the operator $\left(\varepsilon^{2,1}_H\right)^{k_0}$. Recall that $\varepsilon^{2,1}_H$ is the derivative of the right regular action by $E_{2,1}$. It is then clear that it intertwines the left regular action between $\tau_{\eta,\nu}$ and $\tau_{\eta',\nu'}$. We just have to check that it goes between the right spaces. 
    
    Let $ma=\diag(m_1a_1,m_2a_2)\in M_HA_H$ then we have
    \[
    (\varepsilon^{2,1}_H f)(gma)=\left.\frac{d}{dt}\right|_{t=0}f(gma e^{tE_{2,1}})=\left.\frac{d}{dt}\right|_{t=0}f(g e^{tm_1m_2a_1^{-1}a_2E_{2,1}}ma),
    \]
    which is then equal to $|m_1a_1|^{-\nu_1-1-\frac{1}{2}}_{\eta_1+1}|m_2a_2|^{-\nu_2+1+\frac{1}{2}}_{\eta_2+1}(\varepsilon_H^{2,1}f)(g)$. Induction on $k_0$ shows 
    \[
    ((\varepsilon^{2,1}_H)^{k_0} f)(gma)=|m_1a_1|^{-\nu_1-k_0-\frac{1}{2}}_{\eta_1+[k_0]}|m_2a_2|^{-\nu_2+k_0+\frac{1}{2}}_{\eta_2+[k_0]}((\varepsilon_H^{2,1})^{k_0}f)(g).
    \]

    We are left to show that $(\varepsilon_H^{2,1})^{k_0}f$ is invariant under right multiplication by $n\in N_H$. This is based on the following formula for $\ell\in \N$:
    \[
    (\varepsilon^{2,1}_H)^\ell f (gn(y))=\sum_{j=0}^\ell \binom{\ell}{j}(\nu_2-\nu_1-\ell)_j y^j(\varepsilon^{2,1})^{\ell-j} f(g).\qquad (n(y)=\begin{pmatrix}
        1&y\\0&1
    \end{pmatrix})
    \]
    This formula is obtained using
    \[
    (\varepsilon^{2,1}_H)^{\ell+1}f(gn(y))=\left.\frac{d}{dt}\right|_{t=0} (\varepsilon^{2,1}_H)^\ell f\left( g\begin{pmatrix}
        1&0\\ \frac{t}{1+ty}&1
    \end{pmatrix}\begin{pmatrix}
        1+ty&0\\ 0&(1+ty)^{-1}
    \end{pmatrix} n\left(\frac{y}{1+ty}\right)\right),
    \]
    and then using induction on $\ell$.    
    Finally, if $\nu_2-\nu_1=k_0$ we have:
    \[
    (\varepsilon^{2,1}_H)^{k_0} f (gn(y))=(\varepsilon^{2,1}_H)^{k_0} f (g),
    \]
    which ends the proof. 
\end{proof}

\begin{remark}
Notice than when $\nu_2-\nu_1 -[\eta_1+\eta_2]=2j$ the operators from Lemma \ref{lem:epsilonDiffOp} operators are obtained as residues of Knapp--Stein intertwining operators as explained at the end of Section \ref{sec:Knapp-Stein}.
\end{remark}

We are now equipped to prove Proposition \ref{prop:DSBOk=1}.
\begin{proof}[Proof of Proposition \ref{prop:DSBOk=1}]

In case $(1)$ of Proposition \ref{prop:DSBOk=1} a close examination at the formula in Lemma \ref{lemma:OperatorForMultiplicityTwo} shows the desired result.

In case (2) of Proposition \ref{prop:DSBOk=1} according to Lemma \ref{lemma:OperatorForMultiplicityTwo} we have $ B_0\neq 0$ for our set of parameters since the coefficient in front of $\left(\varepsilon^{2,1}\right)^{n_1-\ell_0-1}\left(\varepsilon^{3,2}\right)^{n_2-\ell_0-1}$ is 
    \[
    (-n_1)_{n_1-\ell_0-1}(-n_2)_{n_2-\ell_0-1}\neq 0. 
    \]
Composing after restriction by a non zero differential operator will then give a non zero differential operator. 

Notice that 
\[ 
\left(\varepsilon_H^{2,1}\right)^{\ell_0+1}\circ \rest_1=\rest_1\circ \left(\varepsilon^{3,1}\right)^{\ell_0+1}.  
\]
 Using the fact that $\varepsilon^{3,1}$ commutes with $\varepsilon^{3,2}$ and $\varepsilon^{2,1}$ this gives
\begin{multline*}
\left(\varepsilon_H^{2,1}\right)^{\ell_0+1}\circ B_0=
        (-n_1)_{n_1-\ell_0-1}{\det}_H^{n_1-\ell_0-1}\\
        \times\sum_{i=0}^{n_2-\ell_0-1} (n_1-\ell_0-1-i+1)_i a_i \rest_1\circ \left(\varepsilon^{2,1}\right)^{n_1-\ell_0-1-i}\left(\varepsilon^{3,2}\right)^{n_2-\ell_0-1-i}\left(\varepsilon^{3,1}\right)^{i+\ell_0+1}.
\end{multline*}
As shown in Lemma \ref{lemma:OperatorForMultiplicityTwo}, the renormalized operator $B_1$ is explicitly given by:
\[
    {\det}_H^{n_1} 
        \sum_{i=0}^{k_0} (n_1-i+1)_i a_i \rest_1\circ \left(\varepsilon^{2,1}\right)^{n_1-i}\left(\varepsilon^{3,2}\right)^{k_0-i}\left(\varepsilon^{3,1}\right)^i
\]
with $a_{k_0}=(n_2-k_0+1)_{k_0}.$ Since $i\leq k_0\leq n_1-1$ we have $(n_1-k_0+1)_{k_0}a_{k_0}\neq 0$ and thus the operator is non zero.  Finally we have
\[
 B_1\circ \left(\varepsilon^{3,2}\right)^{n_2-k_0} =
        \sum_{i=0}^{k_0}(n_1-i+1)_i a_i \rest_1\circ \left(\varepsilon^{2,1}\right)^{n_1-i}\left(\varepsilon^{3,2}\right)^{n_2-i}\left(\varepsilon^{3,1}\right)^i.
\]
From their explicit expression one can see that they are not proportional. 
\end{proof}

\subsection{Support at \texorpdfstring{$x_2P_G$}{x2PG}}
As mentioned earlier, we give the results without proofs in this case to avoid repetition with previous subsection.

Let us introduce 
\[n_1:=\nu_1-\lambda_1-\frac{1}{2},\qquad\text{and}\qquad \ n_2:=\nu_1+\nu_2-\lambda_1-\lambda_2-1,\] 
with $N:=\min(n_1,n_2)$. A similar analysis, using a recurrence relation, as for when the support was located at $x_1P_G$ leads to the following.
\begin{proposition}\label{prop:SDBOforIota2}
    We have the following two cases for $\DSBO_2(\pi_{\xi,\lambda},\tau_{\eta,\nu})$:
    \begin{enumerate}
        \item  If $n_1,n_2\in \N$ and if $n_1\leq n_2$ or if there exists $0\leq k_0 \leq n_2$ such that $\lambda_1-\lambda_2+n_1-n_2+k_0=0$, together with
   \[   
    \eta_1=\xi_1+[n_1] , \qquad\text{and}\qquad
    \eta_2=\xi_2+[n_1+n_2] ,
    \]
    then
    \[
    \dim(\DSBO_2(\pi_{\xi,\lambda},\tau_{\eta,\nu}))=1.
    \]
    Moreover, the distribution kernel of a basis of $\DSBO_2(\pi_{\xi,\lambda},\tau_{\eta,\nu})$ is given by: 
\[
\sum_{j=0}^N \frac{(-1)^j(-n_1)_j(\lambda_1-\lambda_2+n_1-n_2+1)_j}{j!} \delta^{(n_1-j)}(x)\delta^{(n_2-j)}(y)\delta^{(j)}(z).
\]
        \item Otherwise
        \[
        \dim(\DSBO_2(\pi_{\xi,\lambda},\tau_{\eta,\nu}))=0. 
        \]
    \end{enumerate}
\end{proposition}

As in the support located at $x_1P_G$ case we also provide a construction of a basis of $\DSBO_2(\pi_{\xi,\lambda},\tau_{\eta,\nu})$ using the operators $\calF_i$, $\calD_i$ and $\varepsilon^{2,1}_H$.
\begin{proposition}  
    For $n_1,n_2\in \N$ and 
    \[   
    \eta_1=\xi_1+[n_1], \qquad\text{and}\qquad 
    \eta_2=\xi_2+[n_1+n_2],
    \] we have two different cases:
    \begin{enumerate}
        \item If $n_1\leq n_2$ then a basis of $ \DSBO_2(\pi_{\xi,\lambda},\tau_{\eta,\nu})$ is given by
    \[
    \rest_2 \circ \calF_1^{n_1}\circ \calF_2^{n_2-n_1}.
    \]
        \item If $n_1>n_2$ and there exists $0\leq k_0 \leq n_2$ such that $\lambda_1-\lambda_2+n_1-n_2+k_0=0$ then a basis of $ \DSBO_2(\pi_{\xi,\lambda},\tau_{\eta,\nu})$ is given by
        \[
        \left(\varepsilon^{2,1}_H\right)^{n_1-k_0}\circ \rest_2 \circ \calF_1^{k_0} \circ \calF_2^{n_2-k_0}.
        \]
    \end{enumerate}
\end{proposition}


\subsection{Support at \texorpdfstring{$x_0P_G$}{x0PG}}
We conclude this section by stating the results for support located at $x_0P_G$. The situation is very similar to the previous case. We redefine:
\[
n_1:=\lambda_3-\nu_2-\frac{1}{2},\qquad\text{and}\qquad n_2:=\lambda_2+\lambda_3-\nu_1-\nu_2-1,\
\]
with $N=\min(n_1,n_2)$. A similar analysis as previously leads to the following result.

\begin{proposition}\label{prop:SDBOforIota0}

We have the following two cases for $\DSBO_0(\pi_{\xi,\lambda},\tau_{\eta,\nu})$:
    \begin{enumerate}
        \item  If $n_1,n_2\in \N$ and if $n_1\leq n_2$ or if there exists $0\leq k_0 \leq n_2$ such that $\lambda_3-\lambda_2-n_1+k_0=0$, together with
   \[   
    \eta_1=\xi_2+[n_1], \qquad\text{and}\qquad
    \eta_2=\xi_3+[n_1+n_2],
    \]
    then
    \[
    \dim(\DSBO_0(\pi_{\xi,\lambda},\tau_{\eta,\nu}))=1.
    \]
    Moreover, the distribution kernel of a basis of $\DSBO_0(\pi_{\xi,\lambda},\tau_{\eta,\nu})$ is given by:
    \[
    \sum_{j=0}^N \frac{(-n_1)_j(\lambda_3-\lambda_2-n_1)_j}{j!} \delta^{(n_2-j)}(x)\delta^{(n_1-j)}(y)\delta^{(j)}(z).
    \]
    \item Otherwise
        \[
        \dim(\DSBO_0(\pi_{\xi,\lambda},\tau_{\eta,\nu}))=0. 
        \]
    \end{enumerate}
\end{proposition}
Finally we also provide a construction of a basis of $\DSBO_0(\pi_{\xi,\lambda},\tau_{\eta,\nu})$ using the operators $\calF_i$, $\calD_i$ and $\varepsilon_H^{2,1}$.
\begin{proposition}  
    For $n_1,n_2\in \N$ and 
    \[   
    \eta_1=\xi_2+[n_1], \qquad\text{and}\qquad
    \eta_2=\xi_3+[n_1+n_2],
    \] we have two different cases:
    \begin{enumerate}
        \item If $n_1\leq n_2$ then a basis of $ \DSBO_0(\pi_{\xi,\lambda},\tau_{\eta,\nu})$ is given by
    \[
    \rest_0 \circ \calD_2^{n_2-n_1}\circ \calD_3^{n_1}.
    \]
        \item If $n_1>n_2$ and there exists $0\leq k_0 \leq n_2$ such that $\lambda_3-\lambda_2-n_1+k_0=0$ then a basis of $ \DSBO_0(\pi_{\xi,\lambda},\tau_{\eta,\nu})$ is given by
        \[
        \left(\varepsilon^{2,1}_H\right)^{n_1-k_0}\circ\rest_0\circ \calD_2^{k_0} \circ \calD_3^{n_2-k_0}. 
        \]
    \end{enumerate}
\end{proposition}

\bibliographystyle{alpha}
\bibliography{biblio}

\end{document}